\documentclass[11pt]{amsart}
  \usepackage[margin=3.3cm]{geometry}
\usepackage{amsfonts,amsmath,amssymb,color,esint}
\usepackage{enumerate}
\usepackage{graphicx}
\usepackage{tikz}
\usetikzlibrary{decorations.markings}
\usetikzlibrary{plotmarks}
\usetikzlibrary{patterns}
\numberwithin{equation}{section}

\usepackage{amsthm}

\usepackage[babel=true,kerning=true]{microtype}

\usepackage{amsthm,leqno}
\usepackage{hyperref}

\newtheorem{theo}{Theorem}[section]
\newtheorem{rem}{Remark}[section]

\newtheorem{cl}{Claim}[section]
\newtheorem{prop}{Proposition}[section]

\newcommand{\eps}{\varepsilon}

\newcommand{\R}{\mathbb{R}}

\begin{document}

\title[]{Maximizing one Laplace eigenvalue on n-dimensional manifolds}
\author{Romain Petrides} 
\address{Romain Petrides, Universit\'e de Paris, Institut de Math\'ematiques de Jussieu - Paris Rive Gauche, b\^atiment Sophie Germain, 75205 PARIS Cedex 13, France}
\email{romain.petrides@imj-prg.fr}

\begin{abstract} 
We prove existence and regularity results for the problem of maximization of one Laplace eigenvalue with respect to metrics of same volume lying in a conformal class of a Riemannian manifold of dimension $n\geq 3$.
\end{abstract}

\maketitle

Let $M$ be a smooth compact manifold of dimension $n$. Given a Riemannian metric $g$ on $M$, we denote the sequence of eigenvalues associated to the Laplace operator $\Delta_g = -div_g \nabla$
$$ 0 = \lambda_0 \leq \lambda_1(g) \leq \lambda_2(g) \leq \cdots \leq \lambda_k(g) \leq \cdots \to +\infty \text{ as } k \to +\infty . $$
The eigenvalue $\lambda_0$ is associated to constant functions and if $M$ is connected, the other ones are positive, depending on $g$. We will focus on the following natural scale invariant functional
$$\bar{\lambda}_k(g) = \lambda_k(g)V_{g}(M)^{\frac{2}{n}}, $$
where $V_{g}(M)$ is the volume of $g$ on $M$. We let $[g] = \{ e^{2u} g ; u \in \mathcal{C}^{\infty}\left(M\right) \}$ be a given conformal class of metrics. In the current paper, we aim at studying the following variational problem
$$ \Lambda_k(M,[g]) = \sup_{\tilde{g} \in [g]} \lambda_k(\tilde{g})V_{\tilde{g}}(M)^{\frac{2}{n}} = \sup_{\tilde{g} \in [g]} \bar{\lambda}_k(\tilde{g}).$$
Studying the maximization of $\bar{\lambda}_k$ is more relevant than the minimization because the well-known examples of Cheeger dumbells prove $\inf_{[g]} \bar\lambda_k = 0$.  Moreover, while it is standard in geometric analysis to restrict a functional with respect to metrics to a conformal class (e.g the Yamabe problem, problems for Q-curvature etc), there are deep reasons to do it in the context of spectral geometry:

On one side, $\Lambda_k(M,[g]) < +\infty$ was proved by Yang-Yau \cite{yangyau} for $k=1$ and $n=2$ and then by Li-Yau \cite{liyau} and El Soufi and Ilias, \cite{EI}, \cite{EI2} for $k= 1$ and $n\geq 2$ thanks to the conformal volume and was generalized by Korevaar \cite{Korevaar} to any $k,n$ (see also \cite{hassann}). On the other side, the supremum of the scale invariant functional $\bar{\lambda}_k(g)$ over all the metrics $g$ on $M$ is always $+\infty$ in dimension $n\geq 3$ \cite{cd}. In dimension $2$, the bound on $\Lambda_k(M,[g])$ is nothing but a tool to obtain an upper bound for $\bar{\lambda}_k(g)$ for any metric $g$ depending only on $k$ and the topology of $M$ (see Yang-Yau \cite{yangyau} for $k= 1$ and Korevaar \cite{Korevaar} for $k\geq 1$). It is then natural to compute $\Lambda_k(M,[g])$, depending on $k$ and the conformal class. As a fundamental result, it was proved by Hersch \cite{hersch} in dimension $2$ and El-Soufi Ilias \cite{EI2} for $n\geq 3$ that 
$$ \Lambda_1( \mathbb{S}^n,[h]) = \bar\lambda_1\left( \mathbb{S}^n, h \right) = n \omega_n^{\frac{2}{n}}$$
where $h$ is the round metric on the $n$-sphere $\mathbb{S}^n$ and $\omega_n$ is its volume, and that the round metric is the unique maximizer. Many other computations of $\Lambda_k(M,[g])$ with existence (or not) of maximizers are given e.g in \cite{nadirashvili}, \cite{ces}, \cite{nadirashvili-2}, \cite{petrides-1}, \cite{petrides}, \cite{petrides-0}, \cite{knpp}, \cite{KS}, \cite{karpukhin-3}, \cite{kim}, and references therein.
 
Furthermore, as observed by Nadirashvili \cite{nadirashvili} and El Soufi-Ilias \cite{EI3}\cite{EI4}\cite{EI5}, critical metrics for $\bar{\lambda}_k$ correspond to harmonic maps into spheres. 
For instance and more precisely, in dimension $2$, the conformal factors $e^{2u}$ of the critical metrics $\tilde{g} = e^{2u} g$ for $\bar{\lambda}_k$ are nothing but densities of energy associated to the equation of harmonic maps into spheres with respect to $g$. 
This observation was crucial to expect regularity of critical metrics constructed by a suitable variational method, by the use of regularity theorems for weakly-harmonic maps.

Finally, studying the variational problem $\Lambda_k(M,[g])$ is a first step to look for critical metrics of $\bar{\lambda}_k(g)$ over all the metrics. As observed by Nadirashvili \cite{nadirashvili} and El Soufi-Ilias \cite{EI3}, critical metrics with respect to variations in the space of 2-symmetric tensors arise as induced metrics of minimal immersions of $M$ into spheres, and conversely, any induced metric of a minimal immersion into a sphere can be seen as the critical metric of one of the functionals $\bar{\lambda}_k$. 
For instance, existence of minimal immersions into sphere by first eigenfunctions is given in \cite{petrides} and \cite{MS3} by maximization of $\lambda_1$ over all the metrics on any surface of dimension 2. While maximization over all the metrics is not possible in dimension $n\geq 3$ \cite{cd}, one approach by min-max was given by Friedlander and Nadirashvili \cite{FN}. 
Their invariant was essentially computed in dimension $2$ in \cite{petrides} \cite{KMed} and is often not realized. Notice that, the variational techniques developped in the current paper and in \cite{petrides-7} are substantial to initiate construction of minimal immersions into spheres by a min-max method on eigenvalues $\bar{\lambda}_k$. 

\medskip

In the past decades, many works have been done in dimension $2$ to compute $\Lambda_k(M,[g])$ and to give methods to prove existence and regularity of maximal metrics, after the seminal work by Nadirashivli on the torus \cite{nadirashvili} and Fraser-Schoen for Steklov eigenvalues on surfaces of genus zero \cite{fs}. In the current paper, we are interested in the generalization to $n\geq 3$ of the study for any $k$ and $M$ with $n=2$ given in \cite{petrides} \cite{petrides-2} and \cite{petrides-7}. They are based on a construction of maximizing sequence of metrics $e^{2u_\eps} g$ and associated maps $\Phi_\eps : M \to \R^{p_\eps}$ which are "almost" harmonic in the following sense:
$$ \Delta_g \Phi_\eps = \lambda_k(e^{2u_\eps} g) e^{2u_\eps} \Phi_\eps $$
and there is $\delta_\eps \to 0$ such that
$$ \left\vert \Phi_\eps \right\vert^2 \geq 1-\delta_\eps \text{ and } \int_M  \left\vert\Phi_\eps \right\vert^2 e^{2u_\eps} = \int_M e^{2u_\eps} = 1.$$
These maps are harmonic into a sphere if and only if $\left\vert \Phi_\eps \right\vert^2 = 1$. In \cite{petrides} and \cite{petrides-2}, this sequence is built by maximization of a regularized functional depending on a parameter $\eps$, and in \cite{petrides-7} we simplify the selection of this maximizing sequence thanks to a new concept of Palais-Smale sequences for the functional $\bar{\lambda}_k$ in the level sets $\bar{\lambda}_k \geq \Lambda_k(M,[g]) - \eps$. The goal is then to pass to the limit as $\eps\to 0$ on $(\Phi_\eps)$ using the elliptic estimates on this super-critical system of equations (in the sense that 
$e^{2u_\eps}$ only belongs to $L^1$ and $\Phi_\eps$ is not uniformly bounded). 
If it is possible, we have that $e^{2u_\eps}dv_g\to \nu$ for the weak-$\star$ topology of measures and 
$\Phi_\eps \to \Phi$,  $\lambda_k(e^{2u_\eps} g)\to \lambda$  so that
$$ \Delta_g \Phi = \lambda \nu \Phi \text{ and } \left\vert \Phi \right\vert^2 = 1 $$
in a weak sence. Computing $0 = \Delta \left\vert \Phi \right\vert^2$ gives that $\nu = \frac{\left\vert \nabla \Phi \right\vert^2}{\lambda} dv_g$ so that the limit equation is the equation of weak harmonic maps, known to be smooth and stongly harmonic, so that the limiting measure is smooth.

\medskip

However, establishing existence and regularity of metrics achieving $\Lambda_k(M,[g])$ in dimension $n\geq 3$ involves other difficulties. The first one is that contrary to dimension $n=2$, if $n\geq 3$ the Dirichlet energy and the Laplacian are not conformally invariant: for a metric $\tilde{g} = e^{2u}g$, we have
$$\lambda_k(\tilde{g}) = \min_{E \in \mathcal{G}_{k+1}\left(\mathcal{C}^\infty(M)\right)} \max_{\varphi \in E\setminus \{0\}} \frac{\int_M \left\vert\nabla \varphi \right\vert_g^2 e^{(n-2)u}dv_g}{\int_M \varphi^2 e^{nu}dv_g} $$
so that the possible degenerescence of maximizing sequence of conformal factors $e^{2u_\eps}$ not only appear at the right-hand side of the elliptic equation for eigenfunctions, as in dimension $2$ (already leading to a super-critical elliptic equations because we only have a $L^1$ control of $e^{nu_\eps}$)
$$ -div_g\left( e^{\left(n-2\right) u_\eps} \nabla \varphi_\eps  \right) = \lambda_k(e^{2u_\eps} g)  e^{nu_\eps} \varphi_\eps, $$
but also at the left-hand side, allowing to loose the elliptic properties of the operator $-div_g\left( e^{\left(n-2\right) u_\eps} \nabla . \right)$ as $\eps\to 0$. The second one is that in dimension $2$ there are bounds on the multiplicity of the eigenvalue $\lambda_k(g)$ depending only on $k$ and the topology $M$, while it is not the case in dimension $n\geq 3$ (see \cite{ColinV}), even with restriction in a conformal class. This boundedness was often used to initiate compactness arguments on the sequence of maximizing metric associated to an "almost critical" system of equations approaching the system of equations of a harmonic map into a sphere. Indeed, the number of equations in the system is automatically uniformly bounded for $n=2$. This is \textit{a priori} not true in higher dimensions.

\medskip

In the following result, we overcome these problems by establishing a natural generalization of the maximization results for $\Lambda_k(M,[g])$ in dimension $2$ to higher dimensions:
\begin{theo} \label{theomain}
Let $(M, [g])$ be a compact connected manifold of dimension $n\geq 3$ endowed with a conformal class and $k\geq 1$. If
$$ \Lambda_k(M,[g]) > \Lambda_k(\widetilde{M},[\widetilde{g}]) $$
for any $(\widetilde{M},[\widetilde{g}]) = (M,[g]) \sqcup  \left(\mathbb{S}^n,[h]\right)\sqcup \cdots \sqcup \left(\mathbb{S}^n,[h]\right) $ or $ (\widetilde{M},[\widetilde{g}]) =  \left(\mathbb{S}^n,[h]\right)\sqcup \cdots \sqcup \left(\mathbb{S}^n,[h]\right) $ where $h$ is the round metric on $\mathbb{S}^n$, then for some $\alpha>0$, there is a non-negative factor $f \in \mathcal{C}^{0,\alpha}\left(M\right) \cap \mathcal{C}^{\infty}(M\setminus Z)$ where $Z = \{ z \in M ; f(z) = 0 \}$ such that 
$$ \lambda_k(f g) \left(\int_M f^{\frac{n}{2}}dv_g\right)^{\frac{2}{n}} = \Lambda_k(M,[g]) $$
Moreover, $f = \frac{\left\vert \nabla \Phi \right\vert_g^2}{\lambda_k(f g)}$, where $\Phi : M \to \mathbb{S}^p$ is some $n$-harmonic map into a sphere, whose coordinate functions are eigenfunctions with respect to $\lambda_k(f g)$.
\end{theo}

Notice that as in dimension $n=2$ (see \cite{petrides} and \cite{petrides-2}), the strict inequality assumptions we make are natural to ensure the compactness of the sequence of critical metrics, and more generally some compactness for maximizing sequences of metrics. These strict inequality assumptions are simple ways to prevent from bubbling that may happen (for more information on the bubble tree convergence in this context, see e.g \cite{petrides-2}, \cite{KS}, \cite{knpp20} for $n=2$ or the more general Theorem \ref{theosplit} and discussions around).

Notice also that as in dimension $n=2$, the conformal factor $f$ appears as the density of energy of a $n$-harmonic map into a sphere
$$ -div_g\left( \left\vert \nabla \Phi \right\vert^{n-2}_g \nabla \Phi \right) = \left\vert \nabla \Phi \right\vert^n_g \Phi. $$
Apparition of $n$-harmonic maps into a sphere in this context was already observed in \cite{KM} (see also \cite{pt}). In fact, $f$ is the limit of a maximizing sequence of conformal factors and we conclude the proof of the theorem by noticing that $f = \frac{\left\vert \nabla \Phi \right\vert_g^2}{\lambda_k(f g)}$, where $\Phi$ is a weak locally minimizing $n$-harmonic map into a sphere. The regularity theory for these maps in the litterature implies $\Phi \in \mathcal{C}^{1,\alpha} $ for some $\alpha\in (0,1)$, and not more. The lack of higher regularity is well known because the weight $\left\vert \nabla \Phi \right\vert^{n-2}_g$ inside the divergence term of the equations may vanish, so that we have a degenerate elliptic system. Therefore $f \in \mathcal{C}^{0,\alpha}$ is the optimal regularity. It is a very common conclusion when we look for conformal metrics as solutions of a variational problem (see e.g in \cite{AmmannHumbert} or \cite{GurskyPerez} the concept of generalized metrics). Notice also that even in dimension $2$, while the zero set $Z$ of $f$ has to be discrete, it may be non-empty. Therefore, even if the conformal factor $f$ is smooth for $n=2$, the associated metric may have conical singularities. Of course, $f \in \mathcal{C}^{\infty}\left(M\setminus Z\right)$ in the domain where the equation of $\Phi$ is elliptic and
 the metric $\tilde{g} = fg$ is regular. Since in $M\setminus Z$, $\left\vert \nabla \Phi \right\vert^2_{\tilde{g}} = \lambda_k(\tilde{g})$ and thanks to conformal invariance of the $n$-harmonic equation, the n-harmonic map into a sphere with respect to $g$ becomes a $2$-harmonic map into a sphere with respect to $\tilde{g} = fg$
$$ \Delta_{\tilde{g}} \Phi = \left\vert \nabla \Phi \right\vert^2_{\tilde{g}} \Phi $$
as was primarily observed by \cite{EI3} for smooth critical metrics. In fact, it is a $p$-harmonic map with respect to $\tilde{g}$ for any $p$, and $p=n$ is the adapted integer that gives a conformally invariant equation. It is a reason why we have to deal with $n$-harmonic maps.

Notice also that it is not clear that generalized metrics $f.g$ have a discrete spectrum and that any eigenvalue has a finite dimensional associated space of eigenfunctions. In the current paper, we actually prove that it holds true as a consequence of
\begin{theo} \label{theodiscretespectrum}
For the factor $f$ given by Theorem \eqref{theomain}, the embedding $W^{1,2}\left( f.g\right) \to L^{2}\left( f.g\right)$ is compact and the eigenfunctions with respect to $\Delta_{f.g}$ are $\mathcal{C}^{1,\alpha}$ functions.
\end{theo}
This compactness property on the weight $f$ generalizes to higher dimensions the concept of "admissible measures" developped in dimension 2 e.g in \cite{KS}. Thanks to this result, we \textit{a posteriori} deduce that the multiplicity of the $k$-th eigenvalue associated to $f.g$ is finite and that the target sphere $\mathbb{S}^p$ in Theorem \ref{theomain} is a finite dimensional sphere. The proof of Theorem \ref{theodiscretespectrum} is based on a new understanding of $\mathcal{C}^{1}$ $n$-harmonic maps: they can be seen as local limits of $(\tau,n)$-harmonic maps, defined as minimizers of the regularized functional
$$ \Psi \mapsto \int \left(\left\vert \nabla \Psi \right\vert^2 + \tau \right)^{\frac{n}{2}}. $$
Then, we can generalize to systems some recent improvements of regularity of the $p$-Laplace equation (see \cite{Sarsa}) by computing Caccioppoli (or reverse H\"older) type inequalities independant of $\tau$ on the equation satisfied by the gradient of $(\tau,n)$-harmonic maps. We can deduce that the $W^{1,2}$ norm of some power of $f$ is controled and we deduce local embeddings $W^{1,2}\left( f.g\right) \to L^{2\kappa_0}\left( f.g\right)$ for some $\kappa_0 > 1$ that imply compact embeddings in $L^p(f.g)$ for $p < 2 \kappa_0$. From this new technique, we also deduce a higher regularity result for $n$-harmonic maps into spheres. 

\medskip

Coming back to our original problem, we know that the strict inequality involved in the assumptions of Theorem  \ref{theomain} always holds true for $k= 1$ if $(M,[g])$ is not equivalent to a sphere (see \cite{petrides-0}). Since by \cite{EI2}, the theorem holds on the conformal class of the round sphere, we then have the existence result for the first eigenvalue:
\begin{theo} \label{theofirsteigenvalue}
Let $(M, [g])$ be a compact connected manifold of dimension $n\geq 3$ endowed with a conformal class, then for some $\alpha>0$, there is a non-negative factor $f \in \mathcal{C}^{0,\alpha}\left(M\right) \cap \mathcal{C}^{\infty}(M\setminus Z)$ where $Z = \{ z \in M ; f(z) = 0 \}$ such that 
$$ \lambda_1(f g) \left(\int_M f^{\frac{n}{2}}dv_g\right)^{\frac{2}{n}} = \Lambda_1(M,[g]) $$
Moreover, $f = \frac{\left\vert \nabla \Phi \right\vert_g^2}{\lambda_1(f g)}$, where $\Phi : M \to \mathbb{S}^p$ is some $n$-harmonic map into a sphere, whose coordinate functions are eigenfunctions with respect to $\lambda_1(f g)$.
\end{theo}
This is the generalization for $n\geq 3$ of the main theorem in \cite{petrides}. 

In fact, for $k > 1$ we always have existence of maximal configurations but they may be a disjoint union of at most $k$ connected surfaces $(\widetilde{M},[\tilde{g}]) = (M,[g]) \sqcup  \left(\mathbb{S}^n,[h]\right)\sqcup \cdots \sqcup \left(\mathbb{S}^n,[h]\right) $ or $ (\widetilde{M},[\widetilde{g}]) =  \left(\mathbb{S}^n,[h]\right)\sqcup \cdots \sqcup \left(\mathbb{S}^n,[h]\right) $ endowed with metrics maximizing lower eigenvalues in their conformal class. This is a consequence of the bubble tree convergence proved in the current paper. For instance bubbling happens for $\Lambda_2(\mathbb{S}^n,[h])$, which is never realized on $\mathbb{S}^n$ \cite{kim}, but realized by a disjoint union of two round spheres of same volume. We add the strict inequality assumption in Theorem \ref{theomain} to be sure to obtain a new maximizer on $M$. Moreover, as a contrapositive, if there is not any maximizer  for $\Lambda_k\left(M,[g]\right)$, we deduce from Theorem \ref{theomain} and a result by Colbois, El Soufi \cite{ces} that the inequality of the assumption of Theorem \ref{theomain} is an equality. For instance, proving that $\Lambda_k\left(\mathbb{S}^n,[h]\right)$ is not realized in the conformal class of a round sphere gives a natural way to prove the following conjecture
$$ \Lambda_k\left(\mathbb{S}^n,[h]\right) = n k^{\frac{2}{n}} \omega_n^{\frac{2}{n}} $$
that is known to be true for $k=1$ (\cite{hersch} and \cite{EI2}), $k=2$ (\cite{nadirashvili-2} \cite{petrides-1} and \cite{kim}) and $n=2$ (\cite{knpp}). This strategy was used in dimension 2 in \cite{knpp} and \cite{karpukhin-3}.

As we will explain in section \ref{splitconformalfactor}, we prove theorem \eqref{theomain} by noticing that the maximization of $\Lambda_k(M,g)$ is in fact the same as the maximization of a more general functional
$$ \bar{\lambda}_k(g,\alpha,\beta) = \lambda_k(g, \alpha, \beta) \frac{ \int_M \beta dv_g }{ \left(\int_M \alpha^{\frac{n}{n-2}}dv_g \right)^{\frac{n-2}{n}}} . $$ 
among non negative functions $\alpha$, $\beta$, where
$$ \lambda_k(g,\alpha,\beta) = \inf_{E \subset \mathcal{G}{k+1}\left(\mathcal{C}^\infty\left(M\right)\right)} \sup_{\phi\in E\setminus\{0\}} \frac{\int_M \left\vert \nabla \phi \right\vert_g^2 \alpha dv_g}{\int_M\phi^2 \beta dv_g}  . $$
Indeed, we have:

\begin{theo} \label{theosplit}
Let $(M, [g])$ be a compact connected manifold of dimension $n\geq 3$ endowed with a conformal class and $k\geq 1$. Then
$$  \Lambda_k(M,[g]) = \sup_{\alpha \geq 0 , \beta \geq 0} \bar{\lambda}_k(g,\alpha,\beta) $$
More precisely, the maximizers $(\widetilde{M},[\tilde{g}]) = (M,[g]) \sqcup  \left(\mathbb{S}^n,[h]\right)\sqcup \cdots \sqcup \left(\mathbb{S}^n,[h]\right) $ or $ (\widetilde{M},[\widetilde{g}]) =  \left(\mathbb{S}^n,[h]\right)\sqcup \cdots \sqcup \left(\mathbb{S}^n,[h]\right) $  of $\bar{\lambda}_k(g, \alpha, \alpha^{\frac{n}{n-2}})$ for $\alpha$ non negative functions  are the same as the maximizers of $\bar{\lambda}_k(g, \alpha, \beta)$ for $(\alpha,\beta)$ a couple of non negative functions.
\end{theo}

Very recently, Karpukhin and Stern \cite{KS2} proposed another variational problem 
$$ \nu_k(M,g) = \sup_{\beta\geq 0} \left(  \lambda_k(g,1,\beta)  \int_M \beta dv_g  \right) $$
which was more convenient to generalize their techniques in dimension 2 \cite{KS} to dimension $n\geq 3$ \cite{KS2}. One reason is that the eigenvalue equations and harmonic map equations associated to critical potentials $\beta$ never become degenerate elliptic system of equations. Notice that the techniques we use in the current paper are adaptable to prove existence and regularity results on $\nu_k(M,g)$ for any $k$. As an immediate consequence of Theorem \ref{theosplit}, we have that
$$ \nu_k(M,g) \leq \Lambda_k(M,[g])V_g(M)^{\frac{n-2}{n}} $$
with equality if and only if $g$ is a maximizer of $\Lambda_k(M,[g])$. For instance, by Theorem \eqref{theofirsteigenvalue}, there is a maximizer for $k=1$. This inequality was only proved for specific conformal classes in \cite{KS2}.

\medskip

Notice that the techniques and results given in the current paper are generalizable to many other eigenvalue problems: combinations of Laplace eigenvalues, or Steklov eigenvalues... This will be written in forecoming papers.

\medskip

The paper is organized as follows: in section \ref{sec1} we first explain the variational approach to prove Theorem \ref{theosplit}. In particular, we define and select a Palais-Smale sequence for this variational problem. In section \ref{sec2}, we prove the bubble tree convergence of Palais-Smale sequences, leading to Theorem \ref{theosplit} and Theorem \ref{theomain}. In particular, we prove at the end of section \ref{sec2} the compactness embedding of Theorem \ref{theodiscretespectrum}. In Section \ref{sec3}, we prove all the necessary regularity results on $n$-harmonic maps into (possibly infinite dimensional) spheres, generalizing the classical $\eps$-regularity results (e.g in \cite{Uhlenbeck}\cite{Strzelecki}\cite{HardtLin}\cite{mouyang}), strong convergence results for sequences of harmonic maps \cite{Cou}, point removability \cite{Uhlenbeck}\cite{HardtLin}\cite{mouyang}...  In particular, inspired by \cite{Uhlenbeck}, we prove \textit{a priori} regularity estimates which are independant of the dimension of the target manifold. By the way, we prove a new result concerning harmonic maps: any $\mathcal{C}^1$ $n$-harmonic map is a locally minimizing harmonic map and is locally a strong limit of minimizing $(\tau,n)$-harmonic maps. This leads to the proof of higher regularity results for $\mathcal{C}^1$ harmonic maps we use for the proof of Theorem \ref{theodiscretespectrum}.

\section{The variational approach} \label{sec1}

\subsection{Splitting the conformal factor into two densities} \label{splitconformalfactor}

Let $(M,g)$ be a Riemannian manifold of dimension $n$. We consider $\alpha$ and $\beta$ two non-negative functions, as weights for the following eigenvalue problem:
$$ \lambda_k(\alpha,\beta) = \inf_{E \subset \mathcal{G}{k+1}\left(\mathcal{C}^\infty\left(M\right)\right)} \sup_{\phi\in E\setminus\{0\}} \frac{\int_M \left\vert \nabla \phi \right\vert_g^2 \alpha dv_g}{\int_M\phi^2 \beta dv_g} $$
Notice that if $\beta = \alpha^{\frac{n}{n-2}}$ is a positive function, $\lambda_k(\alpha,\beta)$ is nothing but the $k$-th Laplace eigenvalue associated to the metric $\beta^{\frac{2}{n}}g$. Thanks to this remark, a natural functional invariant by dilatation to consider is
$$ \bar{\lambda}_k(\alpha,\beta) = \lambda_k(f, h) \frac{ \int_M \beta dv_g }{ \left(\int_M \alpha^{\frac{n}{n-2}}dv_g \right)^{\frac{n-2}{n}}} . $$ 
If the functions $\alpha$ and $\beta$ are smooth positive, by compact embeddings of the natural weighted $L^2$ and $W^{1,2}$ spaces involved in this problem there is existence of eigenfunctions. Let $(\alpha,\beta)$ be such that $\int_M \beta dv_g = 1$ and $\int_M \alpha^{\frac{n}{2}} dv_g = 1$. The subdifferential of $-\bar\lambda_k$ at $(\alpha,\beta)$ (see \cite{pt} for definition and computations of the subdifferential) satisfies
\begin{equation} \label{subdifferential} - \partial\left( - \bar{\lambda}_k\right)(\alpha,\beta) \subset co \left\{ 
 \left( \left\vert \nabla \phi \right\vert_g^2 - \bar\lambda_k \alpha^{\frac{2}{n-2}},  \bar\lambda_k \left( 1 - \phi^2 \right) \right) ; \phi \in E_k(\alpha,\beta), \int_M \phi^2 \beta dv_g = 1
  \right\}
\end{equation}
where $\bar\lambda_k = \bar\lambda_k(\alpha,\beta)$. Notice that as a convention, we compute the subdifferential of $-\bar{\lambda}_k$ because the subdifferential is well adapted for minimization of functionals. If $(\alpha,\beta)$ is critical, we have by definition that $0 \in \partial \bar{\lambda}_k(\alpha,\beta)$. Then, there are eigenfunctions $\Phi = (\phi_1,\cdots, \phi_p)$ such that
$$ \left\vert \nabla \phi \right\vert_g^2 - \bar\lambda_k(\alpha,\beta) \alpha^{\frac{2}{n-2}} = 0 \text{ and } \left\vert \Phi \right\vert^2 = 1$$
Using the system of equations on $\Phi$ (eigenvalue equations)
$$ -div_g(\alpha \nabla \Phi) = \bar{\lambda}_k(\alpha,\beta) \beta \Phi  $$
we obtain by computing $0 =  -div_g( \alpha \nabla \left\vert \Phi \right\vert^2 ) $ that
$$ \alpha \left\vert \nabla \Phi \right\vert^2 = \bar{\lambda}_k(\alpha,\beta) \beta $$
so that $ \beta = \alpha^{\frac{n}{n-2}}$. Then, if a maximizer of $\bar\lambda_k(\alpha,\beta)$ exists, it is a maximizer for the Laplace eigenvalue $\bar\lambda_k$ in a conformal class

\begin{rem} \label{remcompactsubdifferential} Notice that the computation of the subdifferential is valid as soon as we assume that the embedding of the weighted Sobolev spaces involved in the eigenvalue functional $W^{1,2}(\alpha,\beta)  \to L^2(\beta)$ is compact (see \cite{pt}), leading to the existence of eigenfunctions. The maximizers we obtain at the very end of the proof of Theorem \ref{theomain} are not necessarily smooth and positive everywhere but satisfy this compactness embedding.
\end{rem}

\subsection{A deformation Lemma}
We denote by
$$ X = \{(\alpha,\beta) \in \mathcal{C}^0(M)\times \mathcal{C}^0(M), \alpha>0, \beta>0 \} $$ 
Let $(\alpha,\beta)\in X$ with $\int_M \beta dv_g = 1$ and $\int_M \alpha^{\frac{n}{2}} dv_g = 1$. We have that the formula \eqref{subdifferential} holds true on the subdifferential $\partial\bar{\lambda}_k(\alpha,\beta) $
and we set
$$ \left\vert \partial \bar{\lambda}_k(\alpha,\beta) \right\vert =  \max_{\tau\in \bar{X}} \min_{\psi \in \partial \bar\lambda_k(\alpha,\beta)} \left\langle \tau , \psi \right\rangle $$
a pseudo-norm of the subdifferential $ \partial \bar\lambda_k(\alpha,\beta) $, where 
$$\tilde{X} = \{ (\dot{\alpha},\dot{\beta}) \in \mathcal{C}^0\left(M\right) ; \alpha\geq 0, \beta\geq 0 \text{ and } \int_M \dot{\alpha} dv_g + \int_M \dot{\beta} dv_g = 1  \} $$

Notice that $\left\vert \partial \bar{\lambda}_k(\alpha,\beta) \right\vert  \geq 0$ has to be seen as the largest right derivative of $\bar\lambda_k$ among all the variations $v_t = t (\dot{\alpha},\dot{\beta})$ for $\dot{\alpha},\dot{\beta} \in \tilde{X}$.
We let
$$ X_{r} = \{ (\alpha,\beta)\in X ;  \left\vert \partial \bar\lambda_k(\alpha,\beta) \right\vert > 0 \} $$
be the set of the regular points of $f$. The set of critical points is defined as
$$ X_{c} = \{ (\alpha,\beta)\in X ;  \left\vert \partial \bar\lambda_k(\alpha,\beta) \right\vert = 0 \}. $$
Notice that we are looking for critical points in the set $ X_c $ (in the current paper: maximizers)
 
\begin{prop} \label{deformationlemma}
If there is $\eps_0 >0$ and $\delta>0$ such that
$$ \forall (\alpha,\beta) \in X ; \forall \eps\in (0,\eps_0), \left\vert \bar\lambda_k(\alpha,\beta) - c \right\vert \leq \eps \Rightarrow  \left\vert \partial \bar\lambda_k(\alpha,\beta) \right\vert \geq \delta \hskip.1cm, $$
Then $\forall \eps \in (0,\eps_0)$, there is a continuous map $\eta : X \to X$ such that
\begin{itemize}
\item $\eta(\alpha,\beta) = (\alpha,\beta)$ for any $f\in \{ \bar\lambda_k\geq c + \eps_0 \} \cup \{\bar\lambda_k \leq c-\eps_0\}$
\item $\eta(\{ \bar{\lambda}_k \geq c - \eps\}  ) \subset \{ \bar\lambda_k \geq c + \eps\} $
\end{itemize}
\end{prop}

During the proof, we use the notations $E = \bar\lambda_k$ and $x =(\alpha,\beta) \in X $, $\tau = (\dot{\alpha},\dot{\beta})\in \tilde{X}$. We first build an adapted pseudo-vector field for our problem. 

\begin{cl} \label{pseudovf} Let $\eps>0$.  There is a locally Lipschitz vector field $v : X_r \to \bar{X}$ such that for all $x \in X$
\begin{itemize}
\item $ \left\| v(x)\right\|_1 < 2  $
\item $\forall \psi \in \partial E(x) , \left\langle v(x),\psi \right\rangle >  \frac{ \left\vert \partial E(x) \right\vert}{2} $
\item $v(x) \geq 0$
\end{itemize}
\end{cl}

\begin{proof} 
We first fix $x_0 \in X$ and we build an adapted image $v_0(x_0)$ satisfying the conclusion of the Claim.

Let $\tau_0 \in \tilde{X}$ be such that
$$ \left\vert \partial E(x_0) \right\vert \leq \min_{\psi\in \partial E(x_0)} \left\langle \tau_0,\psi \right\rangle + \frac{\left\vert \partial E(x_0) \right\vert}{4} \hskip.1cm.$$

We choose $v_0(x_0) = \tau_0 $ so that 
\begin{itemize}
\item $ \left\| v_0(f_0)\right\|_1 \leq 1 <  2 $
\item $\forall \psi \in \partial E(x_0) , \left\langle v_0(x_0),\psi \right\rangle \geq \frac{3}{4} \left\vert \partial E(x_0) \right\vert >  \frac{\left\vert \partial E(x_0) \right\vert}{2} $
\item $v_0(x_0) \geq 0$ 
\end{itemize}

Now, we aim at defining $v$ by some transformation of $v_0$ in order to obtain a locally Lipschitz vector field $v : X_r \to X$.

Let $x_0 \in X_{r}$. Let $\Omega_{x_0}$ be an open neighborhood of $x_0$ in $X_{r}$ such that for all $x\in \Omega_{x_0}$
$$\forall \psi \in \partial E(x) , \left\langle \psi,v_0(x_0) \right\rangle >  \frac{\left\vert \partial E(x) \right\vert}{2}$$
We notice that 
$$ X_{r} = \bigcup_{x_0 \in X_{r}} \Omega_{x_0}.$$
Since $X_{r}$ is paracompact, one has a family of open sets $(\omega_i)_{i\in I}$ such that
\begin{itemize}
\item $ X_{r} = \bigcup_{i \in I} \omega_i $
\item $\forall i\in I, \exists x_i \in X_{r}, \omega_i \subset \Omega_{x_i}$
\item for all $ u\in X_{r}$ there is an open set $\Omega$ such that $u\in \Omega$ and $\Omega \cap \omega_i = \emptyset$ except for a finite number of indices $i$.
\end{itemize}
We set $\psi_i(u) = d\left(u, X_{r} \setminus \omega_i \right)$ and $\eta_i(u) = \frac{\psi_i}{\sum_{j\in I}\psi_j}$ and the vectorfield
$$ v(x) = \sum_{i\in I} \eta_i(x) v_0(x_i) $$
satisfies the conclusion of the Claim.
\end{proof}

\begin{proof}{(of proposition \ref{deformationlemma})} We define a vector-field $\Phi : X \to X$ as for $x\in X$,
$$ \Phi(x) = \frac{d(x,A)}{d(x,A)+d(x,B)} v(x)$$
where $v : X_r\to X$ is given by Claim \ref{pseudovf} and we define the sets
$$ A = \{ E\leq c-\eps_0\} \cup \{E\geq c+\eps_0\} $$
and 
$$  B = \{ c-\eps \leq E \leq c+\eps\} \hskip.1cm.$$
 Let $\eta$ be a solution of
\begin{equation}
\begin{cases}
\frac{d}{dt}\eta_t(f) = \Phi\left( \eta_t(f)\right) \\
\eta_0(f) = f \hskip.1cm.
\end{cases}
\end{equation}
Such a solution $\eta$ exists since $\Phi$ is locally Lipschitz. Moreover, $\eta$ is well defined on $\mathbb{R}_+$ since $\Phi$ is bounded. Let $t\geq 0$, and $x \in X$. We have thanks to the subdifferential that
\begin{equation}\begin{split} \frac{d}{dt} E(\eta_t(x)) & \geq  \min \{ \left\langle  \Phi(\eta_t(x)), \psi \right\rangle ; \psi \in \partial E\left(\eta_t(x) \right) \} \\
& \geq \frac{d(\eta_t(x),A)}{d(\eta_t(x),A)+d(\eta_t(x),B)} \frac{\left\vert \partial{E}(\eta_t(x)) \right\vert}{ 2 } \geq 0
 \end{split}\end{equation}
for any $t\geq 0$. It is clear that for any $t\geq 0$, $\eta_t(x) = x$ for $x\in A$ and that 
$$\eta_t\left( \{ E\geq c+ \eps \} \right) \subset\{ E\geq c+\eps \} \hskip.1cm.$$
It remains to prove that for $t_0 >0$ small enough we also have that
$$\eta_{t_0}\left( B \right) \subset\{ E\geq c+ \eps \} \hskip.1cm.$$
Let $t_{\star}$ be the smallest time such that $\eta_{t_{\star}}(f) = c + \eps$. Then, for $0\leq t \leq t_{\star}$, we have
$$ \eta_t(x) \in B \Rightarrow  \frac{d}{dt}E(\eta_t(x)) \geq  \frac{\delta}{2} $$
since by assumption, $c - \eps \leq E(\eta_t(x)) \leq c $ implies that $\left\vert \partial{E}(\eta_t(x)) \right\vert \geq \delta$ and that $d(\eta_t(x),B) = 0$. We deduce by integration that
$$ E(\eta_{t_{\star}}(x)) - E(x) \geq  \frac{\delta}{2}t_{\star} $$
so that $t_{\star} \leq \frac{4\eps}{\delta}$. Then, letting $t_0 = \frac{2\eps}{\delta}$ we obtain $\eta_{t_0}(B)\subset \{ E\geq c+\eps \}$. Therefore,
$$\eta_{t_0}(\{ E\geq c- \eps \}) \subset \eta_{t_0}(\{ E\geq c+ \eps \})$$
and we obtain the proposition.
\end{proof}

\subsection{Existence of Palais-Smale sequences for the maximization problem}
If we take $c = \sup_X \bar{\lambda}_k < + \infty$, proposition  \ref{deformationlemma} implies that for any sequence $\eps\to 0$ and $\delta_\eps \to 0$ we have a sequence of $(\alpha_\eps ,  \beta_\eps) \in X$ such that 
$$ \bar\lambda_k(\alpha_\eps ,  \beta_\eps) > c -\eps \text{ and }  \left\vert \partial\bar\lambda_k(\alpha_\eps ,  \beta_\eps)\right\vert < \delta_\eps $$
Now, up to classical convolutions, pairs of smooth positive functions are dense in $X$, so that we have a sequence of pairs smooth positive functions $(\alpha_\eps,e^{nu_\eps} )$ such that
$$ \bar\lambda_k( \alpha_\eps,e^{nu_\eps} ) > c -\eps \text{ and }  \left\vert \partial\bar\lambda_k( \alpha_\eps,e^{nu_\eps} )\right\vert < \delta_\eps $$
the first condition gives that the pair $( \alpha_\eps,e^{nu_\eps} )$ is a maximizing sequence. The second condition can be rewritten as
$$\forall (\dot \alpha, \dot \beta)\in \tilde{X},  \exists \Psi \in -\partial\left( -\bar{\lambda}_k\right)(\alpha_\eps,e^{nu_\eps}) ; \left\langle  (\dot \alpha, \dot \beta) , \Psi \right\rangle \leq \delta_\eps  . $$
We denote $Q_\eps = \lambda_\eps^{\frac{2-n}{2}} \alpha_\eps$ where $ \lambda_\eps =  \bar\lambda_k(\alpha_\eps,e^{nu_\eps})$. 
\begin{equation*}
\begin{split}\forall (\dot \alpha, \dot \beta)\in \tilde{X},  -\partial\left( -\bar{\lambda}_k\right)(\lambda_\eps^{\frac{2-n}{2}}Q_\eps,e^{nu_\eps}) ,
\left\langle  (\dot \alpha, \dot \beta) , (\psi_1-\delta_\eps,\psi_2-\delta_\eps) \right\rangle \leq 0
\end{split} 
\end{equation*}
meaning that
$$ \left( - \partial \left(-\bar\lambda_k\right)( \lambda_{\eps}^{\frac{2-n}{2}}Q_\eps, e^{nu_\eps} ) - \{ ( \delta_\eps, \delta_\eps) \} \right) \cap \{ (a,b) \in \mathcal{C}^0(M) ; a \leq 0, b\leq 0  \} \neq \emptyset.$$ 
Indeed, if not, we use the classical Hahn-Banach theorem to separate these two spaces (the first one is compact, the second one is closed in $\mathcal{C}^0 \times \mathcal{C}^0$) by a linear form $(\mu_1,\mu_2)$ (where $\mu_1$ and $\mu_2$ are Radon measures, belonging to the dual space of continuous functions) satisfying
\begin{equation*}
\begin{split} \forall (\psi_1,\psi_2) \in -\partial\left( -\bar{\lambda}_k\right)(\lambda_\eps^{\frac{2-n}{2}}Q_\eps,e^{nu_\eps}) ,
 \int_M \left( \psi_1 d\mu_1 +\psi_2 d\mu_2 \right) > 0 
 \end{split} 
\end{equation*}
and
$$ \forall (a,b) \in \mathcal{C}^0(M) ; a \leq 0, b\leq 0, \int_M( a d\mu_1 + bd\mu_2) \leq 0 $$
The second condition implies that $\mu_1$ and $\mu_2$ are non negative Radon measures. Up to a renormalization, we assume that $\int_M d\mu_1 + \int_M d\mu_2 = 1$ and we obtain a contradiction. Therefore we obtain the following Palais Smale condition: there is a map $\Phi_\eps = (\phi_1^\eps,\cdots, \phi_{p_\eps}^\eps)$ such that
$$ - div_g\left( \lambda_\eps^{\frac{2-n}{2}}Q_\eps \nabla \Phi_{\eps} \right) =  \lambda_{\eps} e^{nu_{\eps}} \Phi_{\eps} $$
and
$$\int_{M} e^{nu_{\eps}}dA_g = \int_{M} \left\vert \Phi_{\eps} \right\vert^2 e^{nu_{\eps}} dA_g =\lambda_\eps^{-\frac{n}{2}} \int_{M} Q_\eps \left\vert \nabla\Phi_{\eps} \right\vert_g^2 dA_g =  \lambda_\eps^{-\frac{n}{2}} \int_M Q_\eps^{\frac{n}{n-2}} = 1$$
satisfying
$\left\vert \nabla \Phi_\eps \right\vert^2 \leq  Q_\eps^{\frac{2}{n-2}}  + \delta_\eps$ and
$\left\vert \Phi_\eps \right\vert^2 \geq 1 - \frac{\delta_\eps}{\lambda_\eps}$ 
where $\delta_\eps \to 0 $ as $\eps\to 0$.

\section{Convergence of Palais Smale sequences for one Laplace eigenvalue in dimension $n\geq 3$ } \label{sec2}

In the current section, we aim at proving Theorem \ref{theosplit}, as a consequence of the following

\begin{prop} \label{palaissmale}
Let $e^{nu_{\eps}}$, $\lambda_{\eps}$, $\Phi_{\eps} : M \to \mathbb{R}^{p_\eps}$ be a smooth sequence of maps satisfying the Palais-Smale assumption $(PS)$ as $\eps\to 0$, that is
\begin{itemize}
\item $ - div_g\left( \lambda_\eps^{\frac{2-n}{2}}Q_\eps \nabla \Phi_{\eps} \right) =  \lambda_{\eps} e^{nu_{\eps}} \Phi_{\eps} $ and the eigenvalue index of $\lambda_\eps$ is uniformly bounded
\item $\int_{M} e^{nu_{\eps}}dv_g = \int_{M} \left\vert \Phi_{\eps} \right\vert^2 e^{nu_{\eps}} dv_g =\lambda_\eps^{-\frac{n}{2}} \int_{M} Q_\eps \left\vert \nabla\Phi_{\eps} \right\vert_g^2 dv_g =  \lambda_\eps^{-\frac{n}{2}} \int_M Q_\eps^{\frac{n}{n-2}} dv_g = 1$
\item $\left\vert \nabla \Phi_\eps \right\vert_g^2 \leq  Q_\eps^{\frac{2}{n-2}}  + \delta_\eps$ uniformly where $\delta_{\eps}\to 0$ as $\eps\to 0$.
\item $\left\vert \Phi_\eps \right\vert^2 \geq 1 - \delta_\eps$ uniformly, where $\delta_{\eps}\to 0$ as $\eps\to 0$.
\end{itemize}
Then, up to a subsequence and rearrangement of coordinates of $\Phi_\eps$, $p_\eps \to p$, $\lambda_\eps\to \lambda$ and $\Phi_{\eps}$ bubble tree converges in $W^{1,n}$ to $\Phi_0 : M \to \mathbb{S}^p$, $\Phi_j : \mathbb{S}^n \to \mathbb{S}^p$ for $j=1,\cdots,l$ ($l\geq 0$) with an energy identity:
$$ \lambda^{\frac{n}{2}} = \int_{M} \left\vert \nabla \Phi_{0} \right\vert_g^n dv_g + \sum_{j=1}^l  \int_{\mathbb{S}^n} \left\vert \nabla \Phi_{j} \right\vert_h^n dv_h  $$
Moreover, $\Phi_j$ are $\mathcal{C}^{0,\alpha}$ $n$-harmonic maps for $j=0,\cdots,l$ and their $i$-th coordinates are eigenfunctions associated to $\lambda$ on the surface $M \cup \bigcup_{1 \leq i\leq l} \mathbb{S}^n$ endowed with the generalized metrics $\frac{\left\vert \nabla \Phi_0 \right\vert^2}{\lambda}g$ on $M$ and $\frac{\left\vert \nabla \phi \right\vert^2}{ \lambda }h$ on $\mathbb{S}^n$. 
\end{prop}

All along the proof, every local computation is made in the exponential chart centered at points $p \in M$, defined on balls whose radius is controlled by the injectivity radius with respect to $g$: $inj_g(M)$. Without loss of generality, we can assume that $inj_g(M) \geq 2$ and make arguments on the unit ball $\mathbb{B}_n$ endowed with the metric $\exp_{p}^\star g$ still denoted $g$. We do not change the notations of the metrics and functions in the charts. Moreover, when there is not any embiguity, we do not precise the measures of integration associated to $g$, $dv_g$ inside the integrals in order to lighten the computations.

\subsection{A first bubble tree structure}

For a Radon measure, $\mu$ on $M$ (or $\mathbb{R}^n$), we denote $\tilde{\mu}$ the continuous part of $\mu$, defined as the measure $\tilde{\mu}$ which does not have any atom and such that $\mu - \tilde{\mu}$ is a (maybe infinite) linear combination of Dirac masses.
\begin{cl} \label{clbubbletree} Let $Q_\eps$, $e^{nu_\eps}$ be sequences of smooth positive weights such that
$\int_{M} e^{nu_{\eps}}dv_g =  \int_M Q_\eps^{\frac{n}{n-2}} dv_g = 1$ and
$$ \liminf_{\eps\to 0} \lambda_k(Q_\eps, e^{nu_\eps}) > 0 $$
Then up to a subsequence on $\eps\to 0$, there are an integer $t\geq 0$ and if $t>0$ sequences of  points $q_\eps^i$ for $i= 1,\cdots, t$ and scales $\alpha_\eps^t \leq \cdots \leq \alpha_\eps^1 \to 0$ as $\eps \to 0$ such that we have the weak-${\star}$ convergence in $M$, of $ e^{nu_\eps}$  to a non negative Radon measure $\mu_0$ in $M$, we have the weak-${\star}$ convergence in (any compact set of) $\mathbb{R}^n$, of $e^{n u_\eps\left(  \alpha_\eps^i z + q_\eps^i \right)} $ to a non negative and non zero Radon measure $\mu_i $ on $\mathbb{R}^n$ for any $i$. Moreover, their associated continuous parts preserve the mass before the limit:
$$\lim_{\eps\to 0} \int_M e^{nu_\eps}dv_g = \int_M d\tilde{\mu}_0 + \sum_{i=1}^{t} \int_{\R^n} d\tilde{\mu}_i = 1 \text{ and } \forall i\geq 1, \int_{\R^n} d\tilde{\mu}_i >0 $$
and we have that $1 \leq \mathbf{1}_{\tilde{\mu}_0 \neq 0} + t \leq k$. Finally if we set $F_i = \{ j> i ; \left(\frac{d_g(q_i^\eps,q_j^\eps)}{\alpha_i^\eps}\right) \text{ is bounded} \}$, we have for $j>i$,
$$ j\in F_i \Rightarrow \frac{\alpha_j^\eps}{\alpha_i^\eps} \to 0 \text{ and } j\notin F_i \Rightarrow \frac{d_g(q_i^\eps,q_j^\eps)}{\alpha_i^\eps} \to +\infty $$
\end{cl}
In other words, the last condition reads as
$$\frac{\alpha_i^\eps}{\alpha_j^\eps} + 
\frac{\alpha_j^\eps}{\alpha_i^\eps} + 
\frac{d_g(q_i^\eps,q_j^\eps)}{\alpha_i^\eps + \alpha_j^\eps} 
\to +\infty $$
which is the standard condition for a bubble tree. The aim of the rest of section \ref{sec2} is devoted to prove that the continuous measures $\tilde\mu_i$ for $i\in \{0,\cdots,t\}$ are absolutely continuous with respect to $dv_g$ if $i= 0$ and to the Euclidean metric if $i\geq 1$, with densities equal to densities of energy of $n$-harmonic maps.
 We choose to skip the proof of Claim \ref{clbubbletree} because it is a simple adaptation of standard arguments of dimension 2 to higher dimensions we can find in \cite{kokarev}, \cite{petrides}, \cite{petrides-2}, \cite{KS}, \cite{knpp20}. The selection of scales of concentration is based on Hersch's trick because it uses the conformal group of the sphere to balance continuous measures on a sphere. Then the selection stops because of the min-max characterization of $\lambda_k$: If there are more than $k+1$ scales of concentration, we can build $k+1$ independant test functions with arbitrarily small rayleigh quotient, contradicting the first assumption of Claim \ref{clbubbletree}.

\subsection{Some convergence of $\omega_{\eps}$ to $1$ and first replacement of $\Phi_{\eps}$}
We set $\omega_{\eps} = \left\vert \Phi_{\eps} \right\vert$. We first prove that in some sense, $\omega_{\eps}$ converges to $1$ and that $\Phi_{\eps}$ has a similar $H^1$ behaviour as $\frac{\Phi_{\eps}}{\omega_{\eps}}$

\begin{cl} We have that
\begin{equation}  \label{eqomegaepsto1dimn}  \int_{M} Q_\eps \frac{\left\vert \nabla \omega_\eps \right\vert_g^2}{\omega_\eps}dv_g 
= O(\delta_{\eps}) \end{equation}
\begin{equation} \label{eqradialreplacementdimn} \int_{M}  Q_\eps \left\vert \nabla\left( \Phi_\eps - \frac{\Phi_\eps}{\omega_\eps} \right) \right\vert_g^2 dv_g = O\left( \delta_\eps \right) \end{equation}
and for all $q\leq n$, and for any sequence of functions $f_\eps$ in $L^{\frac{n}{n-q}}\left(M,g\right)$,
\begin{equation} \label{eqQepsPhiepsdimn}  \int_{M}  Q_\eps^{\frac{q}{n-2}} f_\eps dv_g = \int_{M} \left\vert \nabla \Phi_\eps \right\vert_g^q f_\eps dv_g + \| f_{\eps} \|_{L^{\frac{n}{n-q}}} O\left(\delta_\eps^{\frac{2}{n}}\right) \end{equation}
as $\eps\to 0$.
\end{cl}

\begin{proof}

We integrate $ - div_g\left(\lambda_{\eps}^{\frac{2-n}{2}} Q_\eps \Phi_{\eps} \right) =  \lambda_{\eps} e^{nu_{\eps}} \Phi_{\eps} $ against $ \Phi_\eps$ and $\frac{ \Phi_\eps}{\omega_\eps}$. We obtain
$$ \int_{M}\left\vert  \Phi_\eps \right\vert^2 e^{nu_\eps} = \lambda_\eps^{-\frac{n}{2}}\int_{M} Q_\eps \left\vert \nabla\Phi_\eps \right\vert^2 $$
and
\begin{equation*} \begin{split}
\int_{M} \frac{\left\vert \Phi_\eps \right\vert^2}{\omega_\eps} e^{n u_\eps} = & - \lambda_{\eps}^{-\frac{n}{2}} \int_{M}\frac{\Phi_\eps}{\omega_\eps}div_g\left(Q_\eps \nabla \Phi_{\eps} \right) \\
= & \lambda_{\eps}^{-\frac{n}{2}} \int_{M} Q_\eps \left(\frac{\left\vert \nabla\Phi_\eps \right\vert^2}{\omega_\eps} - \frac{\left\vert \nabla\omega_\eps \right\vert^2}{\omega_\eps}\right) \\
= & \lambda_{\eps}^{-\frac{n}{2}} \int_{M}\omega_\eps Q_\eps \left\vert \nabla \frac{\Phi_\eps}{\omega_\eps} \right\vert^2 \end{split} \end{equation*}
Therefore
\begin{equation*} 
\begin{split}
\int_{M}Q_\eps \frac{\left\vert \nabla \omega_\eps \right\vert^2}{\omega_\eps}dv_g  = & \int_{M}Q_\eps  \left(\frac{\left\vert \nabla \Phi_\eps  \right\vert^2}{\omega_\eps} -\omega_\eps \left\vert \nabla \frac{\Phi_\eps}{\omega_\eps} \right\vert^2 \right)dv_g   \\ 
 = & \int_{M} \left( Q_\eps \frac{\left\vert \nabla \Phi_\eps  \right\vert^2}{\omega_\eps} - \lambda_\eps^{\frac{n}{2}} e^{nu_\eps}\frac{\left\vert \Phi_\eps \right\vert^2}{\omega_\eps} \right)dv_g \\
 = & \int_{M}\left(\frac{Q_\eps\left\vert \nabla \Phi_\eps  \right\vert^2}{\omega_\eps} -Q_\eps\left\vert \nabla \Phi_\eps  \right\vert^2 \right) - \lambda_\eps^{\frac{n}{2}} \int_{M} e^{nu_\eps} \left(\frac{\left\vert  \Phi_\eps \right\vert^2}{\omega_\eps}  - \left\vert \Phi_\eps \right\vert^2 \right) \\
 = & \int_{M}\frac{Q_\eps\left\vert \nabla \Phi_\eps  \right\vert^2}{\omega_\eps} \left( 1 - \omega_\eps \right) +  \lambda_\eps^{\frac{n}{2}}\int_{M} e^{nu_\eps}(1 -\omega_{\eps}) \\
 \end{split}
\end{equation*}
noticing the crucial equality $\int_{M} e^{nu_\eps}\omega_\eps^2 = \int_{M} e^{nu_\eps} = 1 $. We know that $ \omega_\eps - 1  \geq \sqrt{1-\delta_{\eps}} - 1$ so that
\begin{equation*} 
\begin{split}
\int_{M}Q_\eps \frac{\left\vert \nabla \omega_\eps \right\vert^2}{\omega_\eps}dv_g
\leq &  \int_{M}Q_\eps \left\vert \nabla \Phi_\eps  \right\vert^2  \frac{1- \sqrt{1-\delta_\eps}}{\sqrt{1-\delta_\eps}} +  \lambda_\eps^{\frac{n}{2}} \left(1-\sqrt{1-\delta_{\eps}} \right) = O\left( \delta_\eps \right)
 \end{split}
\end{equation*}
and we obtain \eqref{eqomegaepsto1dimn}. 
Let's prove \eqref{eqradialreplacementdimn}. We have that
$$ \nabla\left( \Phi_\eps - \frac{\Phi_\eps}{\omega_\eps} \right) = \left(1-\frac{1}{\omega_\eps}\right) \nabla \Phi_\eps + \frac{\nabla \omega_\eps}{\omega_\eps^2} \Phi_\eps$$
and then
$$ \left\vert \nabla\left( \Phi_\eps - \frac{\Phi_\eps}{\omega_\eps}\right) \right\vert^2 = \left\vert \nabla \Phi_\eps  \right\vert^2 \left( 1-\frac{1}{\omega_\eps} \right)^2 +  \frac{\left\vert \nabla \omega_\eps  \right\vert^2}{\omega_\eps^2} + 2 \frac{\left\vert \nabla \omega_\eps  \right\vert^2}{\omega_\eps} \left( 1-\frac{1}{\omega_\eps} \right) $$
so that 
\begin{equation*} 
\int_{M}Q_\eps \left\vert \nabla\left( \Phi_\eps - \frac{\Phi_\eps}{\omega_\eps} \right) \right\vert^2 \leq \int_{M} Q_\eps \left\vert \nabla \Phi_\eps \right\vert_g^{2}  \left( 1-\frac{1}{\omega_\eps} \right)^2  + O\left( \delta_\eps\right)
\end{equation*}
as $\eps\to 0$ and
\begin{equation*} 
\begin{split}
\int_{M} Q_\eps \left\vert \nabla \Phi_\eps  \right\vert^2 & \left( 1-\frac{1}{\omega_\eps} \right)^2 =  - \int_{\Sigma} div\left(  Q_\eps \nabla \Phi_\eps  \left( 1-\frac{1}{\omega_\eps} \right)^2 \right) \Phi_\eps \\
= & \lambda_\eps^{\frac{n}{2}} \int_{M}   \left( 1-\frac{1}{\omega_\eps} \right)^2 \left\vert  \Phi_\eps \right\vert^2e^{nu_\eps} + 2\int_{M}Q_\eps \nabla\left(\frac{1}{\omega_\eps}\right)  \omega_\eps \nabla \omega_\eps \left(1-\frac{1}{\omega_\eps}\right) \\
= &  \lambda_\eps^{\frac{n}{2}} \int_{M} e^{nu_\eps} \left(\omega_\eps^2 + \frac{1}{\omega_\eps} - 2\right)  - 2\int_{M} Q_\eps \frac{\left\vert \nabla \omega_\eps \right\vert^2}{\omega_\eps^2} \left(\omega_\eps-1\right)  \\
= &  \lambda_\eps^{\frac{n}{2}} \int_{M} e^{2u_\eps} \left(\frac{1-\omega_\eps}{\omega_\eps}  \right) dv_g  + O\left( \delta_\eps^{2}\right) = O\left(\delta_\eps\right) \\
 \end{split}
\end{equation*}
as $\eps\to 0$, where we crucially used again $ \int_{\Sigma} e^{2u_\eps}\omega_\eps^2 = \int_{\Sigma} e^{2u_\eps} = 1$. We then obtain \eqref{eqradialreplacementdimn}. Now, let's prove \eqref{eqQepsPhiepsdimn}. We define $A_\eps$ by
$$ A_\eps^2 = Q_\eps^{\frac{2}{n-2}} +\delta_\eps - \left\vert \nabla \Phi_\eps \right\vert^2 $$
$A_\eps$ is a non-negative function and we know that 
$$ \int_M \left\vert \nabla \Phi_\eps \right\vert^2_g Q_\eps dv_g =  \int_M Q_\eps^{\frac{n}{n-2}} dv_g = 1. $$
Therefore multiplying $A_\eps^2$ by $Q_\eps^{\frac{n}{n-2}}$ and integrating, we obtain
$$ \int_M A_\eps^2 Q_\eps dv_g = O\left( \delta_\eps \right) \text{ as }\eps\to 0.$$
Now, we also obtain that
$$ \int_M A_\eps^2 \left( \left\vert \nabla \Phi_\eps \right\vert^2 + A_\eps^2\right)^{\frac{n-2}{2}} = \int_M A_\eps^2\left(Q_\eps^{\frac{2}{n-2}}+ \delta_\eps\right)^{\frac{n-2}{2}} \leq c_n \left( \int_M A_\eps^2 Q_\eps + \delta_\eps^{\frac{n-2}{2}} \int_M A_\eps^2 \right) $$
for a constant $c_n$ depending only on the dimension so that by a minoration of the left-hand term we obtain
$$ \int_M A_\eps^2 \left\vert \nabla \Phi_\eps \right\vert_g^2 + \int_M A_\eps^n = O\left( \delta_\eps + \delta_\eps^{\frac{n-2}{2}} \int_M A_\eps^2\right) $$
and in particular, we deduce
$$ \int_M A_\eps^n = O\left( \delta_\eps\right) $$
Now, writing
$$ Q_\eps^{\frac{q}{n-2}} - \left\vert \nabla \Phi_\eps \right\vert^{\frac{q}{2}}  =  Q_\eps^{\frac{q}{n-2}} - \left(Q_\eps^{\frac{2}{n-2}} + \delta_\eps\right)^{\frac{q}{2}} + \left(A_\eps^2 + \left\vert \nabla \Phi_\eps \right\vert^2\right)^{\frac{q}{2}} - \left\vert \nabla \Phi_\eps \right\vert^{\frac{q}{2}}  $$
we obtain that
$$ \int_M \left(Q_\eps^{\frac{q}{n-2}} - \left\vert \nabla \Phi_\eps \right\vert^{\frac{q}{2}}\right) f_\eps = \| f_{\eps} \|_{L^{\frac{n}{n-q}}} \left(O\left(\delta_\eps\right) + O\left(\delta_\eps^{\frac{2}{n}}\right) \right) $$
which completes the proof of the claim.
\end{proof}

From the previous claim and assumptions on the Palais-Smale sequence, we already deduce a global $W^{1,q}$ convergence of $\frac{\Phi_\eps}{\omega_\eps}$ for any $q< n$.

\begin{cl} \label{clW1qforPhieps} Up to a subsequence, there is a map $\Phi_0 : M \to \mathbb{S}^{\mathbb{N}}$ such that
$$ \left\vert \Phi_0 \right\vert^2 :=  \sum_{i=1}^{+\infty} \left(\phi_0^i\right)^2 =_{a.e} 1  \text{ and } \int_M \left\vert \nabla \Phi_0 \right\vert_g^n dv_g  \leq_{\eps \to 0} \limsup \int_M \left\vert \nabla \frac{\Phi_\eps}{\omega_\eps} \right\vert_g^n dv_g $$
where $\left\vert \nabla \frac{\Phi_\eps}{\omega_\eps}  \right\vert_g^2 :=  \sum_{i=1}^{+\infty} \left\vert \nabla\frac{\phi_\eps^i}{\omega_\eps}\right\vert_g^2$ and such that for any $1\leq p < +\infty$, $1\leq q < n$,
$$ \int_M \left( \left\vert \frac{\Phi_\eps}{\omega_\eps} -\Phi_0  \right\vert^p  + \left\vert \nabla \left(\frac{\Phi_\eps}{\omega_\eps}-\Phi_0 \right) \right\vert_g^{q} \right)dv_g \to 0 $$
as $\eps \to 0$.
\end{cl}

\begin{proof} We use Claim \ref{clW1qconvergence} in the appendix where we extend $\Phi_\eps$ by $\phi_\eps^i = 0$ for $i \geq p_\eps + 2$. We have that
$$ -div_g\left( \left\vert \nabla\frac{ \Phi_\eps}{\omega_\eps} \right\vert^{n-2}\nabla \frac{\Phi_\eps}{\omega_\eps}  \right) = \lambda_\eps^{\frac{n}{2}} e^{nu_\eps} \Phi_\eps - div_g\left(  \left\vert \nabla\frac{ \Phi_\eps}{\omega_\eps} \right\vert^{n-2}\nabla \frac{\Phi_\eps}{\omega_\eps} - Q_\eps \nabla \Phi_\eps \right) $$
We notice that $A_\eps = \lambda_\eps^{\frac{n}{2}} e^{nu_\eps} \Phi_\eps $ satisfies $\left(\left\vert A_\eps \right\vert \right)_\eps $ is bounded in $L^1$ and it remains to prove that
$$ B_\eps = - div_g\left(  \left\vert \nabla\frac{ \Phi_\eps}{\omega_\eps} \right\vert^{n-2}\nabla \frac{\Phi_\eps}{\omega_\eps} - Q_\eps \nabla \Phi_\eps \right) $$
satisfies $\| B_{\eps} \|_{W^{-1,n}} \to 0$ as $\eps \to 0$. Let $\eta : M \to \mathbb{R}^{\mathbb{N}}$ be such that
$$ \int_M  \left( \left\vert \eta  \right\vert^n  + \left\vert \nabla\eta \right\vert_g^{n} \right)dv_g < +\infty. $$
We have by an integration by parts that
\begin{equation*}
\begin{split} \int_M B_\eps . \eta dv_g  = & \int_M \nabla \eta \left( \left\vert \nabla\frac{ \Phi_\eps}{\omega_\eps} \right\vert^{n-2}\nabla \frac{\Phi_\eps}{\omega_\eps} - Q_\eps \nabla \Phi_\eps \right) \\
= & \int_M \left( \left\vert \nabla\frac{ \Phi_\eps}{\omega_\eps} \right\vert^{n-2} -  \left\vert \nabla\Phi_\eps \right\vert^{n-2} \right) \nabla \eta \nabla \frac{\Phi_\eps}{\omega_\eps} + \int_M \left(  \left\vert \nabla\Phi_\eps \right\vert^{n-2} - Q_\eps \right) \nabla \eta \nabla \frac{\Phi_\eps}{\omega_\eps} \\
& + \int_M Q_\eps \nabla \eta .  \nabla\left( \Phi_\eps - \frac{\Phi_\eps}{\omega_\eps}\right) dv_g 
\end{split}
\end{equation*}
The second right-hand term satisfies by \eqref{eqQepsPhiepsdimn} and then a H\"older inequality
$$ \left\vert \int_M \left(  \left\vert \nabla\Phi_\eps \right\vert^{n-2} - Q_\eps \right) \nabla \eta \nabla \frac{\Phi_\eps}{\omega_\eps} \right\vert \leq O\left( \delta_\eps^{\frac{2}{n}} \right) \| \left\vert \nabla \eta \right\vert \|_{L^n} . $$
The third right-hand term satisfies by a H\"older inequality, then \eqref{eqradialreplacementdimn}, and then another H\"older inequality,
$$ \left\vert \int_M Q_\eps \nabla \eta .  \nabla\left( \Phi_\eps - \frac{\Phi_\eps}{\omega_\eps}\right) dv_g \right\vert \leq O\left( \delta_\eps \right) \| \left\vert \nabla \eta \right\vert \|_{L^n} . $$ 
we write the first right-hand term as
$$ \left\vert \int_M \left( \left\vert \nabla\frac{ \Phi_\eps}{\omega_\eps} \right\vert^{n-2} -  \left\vert \nabla\Phi_\eps \right\vert^{n-2} \right) \nabla \eta \nabla \frac{\Phi_\eps}{\omega_\eps} \right\vert \leq c_n  \int_M \left\vert \nabla\left( \Phi_\eps - \frac{\Phi_\eps}{\omega_\eps} \right) \right\vert_g \left\vert \nabla \Phi_\eps \right\vert_g^{n-2} \left\vert \nabla \eta \right\vert_g  $$
for a constant $c_n$ depending only on the dimension $n$ where we used that
$$ \left\vert \nabla \frac{\Phi_\eps}{\omega_\eps} \right\vert_g^2 \leq \frac{\left\vert \nabla \Phi_\eps \right\vert_g^2}{1-\delta_\eps} \leq 2 \left\vert \nabla \Phi_\eps \right\vert_g^2 $$
for $\eps$ small enough. Then, we have
\begin{equation*} 
\begin{split}
\int_M \left\vert \nabla\left( \Phi_\eps - \frac{\Phi_\eps}{\omega_\eps} \right) \right\vert_g \left\vert \nabla \Phi_\eps \right\vert_g^{n-2} \left\vert \nabla \eta \right\vert_g = & \int_M \left\vert \nabla\left( \Phi_\eps - \frac{\Phi_\eps}{\omega_\eps} \right) \right\vert_g \left( \left\vert \nabla \Phi_\eps \right\vert_g^{n-2} - Q_\eps \right) \left\vert \nabla \eta \right\vert_g \\
&+ \int_M \left\vert \nabla\left( \Phi_\eps - \frac{\Phi_\eps}{\omega_\eps} \right) \right\vert_g Q_\eps \left\vert \nabla \eta \right\vert_g 
\end{split}
\end{equation*}
and as before, we apply \eqref{eqQepsPhiepsdimn} and then a H\"older inequality for the first term and H\"older inequalities and \eqref{eqradialreplacementdimn} for the second term to obtain
$$ \left\vert \int_M \left( \left\vert \nabla\frac{ \Phi_\eps}{\omega_\eps} \right\vert^{n-2} -  \left\vert \nabla\Phi_\eps \right\vert^{n-2} \right) \nabla \eta \nabla \frac{\Phi_\eps}{\omega_\eps} \right\vert \leq O\left( \delta_\eps^{\frac{2}{n}} \right) \| \left\vert \nabla \eta \right\vert \|_{L^n}  $$
and gathering all the previous inequalities gives the Claim.
\end{proof}
In the next subsections, we aim at proving that the strong convergence holds in a better space, in order to pass to the limit on the sequence of equations satisfied by $\Phi_\eps$.

\subsection{Regularity of the limiting measures}

We choose to prove in detail the regularity of $\tilde\mu_0$ on $M$, meaning that it is absolutely continous with respect to $dv_g$ and that the density is $\mathcal{C}^{0,\alpha}$ as the density of energy of a $n$-harmonic map. 

The proof of the regularity of the measures $\tilde\mu_i$ in $\mathbb{R}^n$ is similar, because locally, all the regularity estimates we make hold for any rescalings centered at $q_\eps^i$ with scale $\alpha_\eps^i$ of the involved functions. Just notice that since the metric rescales as $g\left(\alpha_\eps^i\left(z - q_\eps^i\right)\right)$, it converges to a Euclidean metric as $\eps\to 0$. Up to a reverse stereographic projection and thanks to the point removability property of $n$-harmonic maps, the regular measures in $\mathbb{R}^n$ give $\mathcal{C}^{0,\alpha}$ conformal factors for the round sphere.

\subsubsection{Selection of bad points}
In the following, we perform local regularity estimates on $(\Phi_\eps)$. These estimates can only be done far from "bad points" we select in Claim \ref{clbadpoints}.
For $\Omega \subset M$ a domain of $M$, we set
$$ \lambda_\star(\Omega, e^{nu_\eps}, Q_\eps ) = \inf_{\varphi \in \mathcal{C}_c^{\infty}\left(\Omega\right) } \frac{\int_\Omega  \left\vert \nabla \varphi \right\vert_g^2 Q_\eps dv_g }{\int_\Omega \varphi^2 e^{nu_\eps}dv_g }. $$
Then, knowing that $\lambda_\eps$ is a $k$-th eigenvalue on $M$ we have:
\begin{cl} \label{clbadpoints} Up to a subsequence, there is $0< r_\star <1$ and a set of at most $k$ bad points $B = \{p_1,\cdots, p_s\} \subset M$ and  such that for any $p \in M\setminus\{p_1,\cdots,p_s\}$ and any $r < \min\left( r_\star, d_g(x,\{p_1,\cdots,p_s\} ) \right) $, then
$$ \lambda_\star\left( \mathbb{B}_r(p), e^{nu_\eps}, Q_\eps \right) \geq \lambda_\eps^{\frac{2}{n}} . $$
\end{cl}
Notice that the atoms of $\mu_0$ belong to this set of points $\{p_1,\cdots,p_s\}$.
\begin{proof}
We set 
$$ r_\eps^1 =  \inf\{ r>0 ; \exists p\in M, \lambda_\star\left( \mathbb{B}_r(p), e^{nu_\eps}, Q_\eps \right) < \lambda_\eps^{\frac{2}{n}} \}. $$
If $r_\eps^1$ does not converge to $0$, then up to a subsequence, there is $r_\star$ such that $r_\eps^1 > r_\star$ and Claim \ref{clbadpoints} is proved for this $r_\star$ and $B= \emptyset$.
If $r_\eps^1 \to 0$, then we let $p_1^\eps$ be such that $ \lambda_\star\left( \mathbb{B}_{r_\eps^1}(p^1_\eps), e^{nu_\eps}, Q_\eps \right) = \lambda_\eps^{\frac{2}{n}} $.
By induction assume that for $j \in \mathbb{N}$ we constructed $r_\eps^1 \leq r_\eps^2 \leq \cdots \leq r_\eps^{j-1}$ such that $r_\eps^{j-1}\to 0$ and points $p^1_\eps,\cdots,p^{j-1}_\eps$ such that 
$$ \forall i_1 \neq i_2, \mathbb{B}_{r^i_\eps}(p_i^\eps) \cap \mathbb{B}_{r^i_\eps}(p^i_\eps) = \emptyset  \text{ and } \forall i, \lambda_\star\left( \mathbb{B}_{r_\eps^i}(p^i_\eps), e^{nu_\eps}, Q_\eps \right) = \lambda_\eps^{\frac{2}{n}}$$
then we set 
$$ r_\eps^j = \inf\{ r>0 ; \exists p\in M,  \forall i, \mathbb{B}_r(p) \cap  \mathbb{B}_{r^i_\eps}(p^i_\eps) = \emptyset  \text{ and } \lambda_\star\left( \mathbb{B}_r(p), e^{nu_\eps}, Q_\eps \right) < \lambda_\eps^{\frac{2}{n}} \} $$
Then if $r_\eps^j $ does not converge to $0$ and up to a subsequence, there is $r_\star$ such that $r_\eps^j > r_\star $ and the proof of Claim  \ref{clbadpoints} is proved for this $r_\star$ and $B= \{p_1,\cdots, p_{j-1}\}$ where up to a subsequence we took $p_1,\cdots, p_{j-1}$ as limits of $p_1^\eps,\cdots, p_{j-1}^\eps$ as $\eps \to 0$. If $r_\eps^j \to 0$, then let $p_j^\eps$ be such that  $ \lambda_\star\left( \mathbb{B}_{r_\eps^j}(p^1_\eps), e^{nu_\eps}, Q_\eps \right) = \lambda_\eps^{\frac{2}{n}} $ and $\mathbb{B}_{r_\eps^j}(p_\eps^j) \cap \mathbb{B}_{r_\eps^j}(p_\eps^i)=\emptyset$ for $i<j$.

This induction process has to stop because if we have we constructed $r_\eps^1 \leq r_\eps^2 \leq \cdots \leq r_\eps^{k+1}$ such that $r_\eps^{k+1}\to 0$ and points $p^1_\eps,\cdots,p^{k+1}_\eps$ such that 
$$ \forall i_1 \neq i_2, \mathbb{B}_{r^i_\eps}(p_i^\eps) \cap \mathbb{B}_{r^i_\eps}(p^i_\eps) = \emptyset  \text{ and } \forall i, \lambda_\star\left( \mathbb{B}_{r_\eps^i}(p^i_\eps), e^{nu_\eps}, Q_\eps \right) = \lambda_\eps^{\frac{2}{n}}$$
then, using the first eigenfunction $\varphi_i$ associated to the eigenvalue $\lambda_\star\left( \mathbb{B}_{r_\eps^i}(p^i_\eps), e^{nu_\eps}, Q_\eps \right)$ extended by $0$ in $M \setminus \mathbb{B}_{r_\eps^i}(p^i_\eps)$, we have by the min-max characterization of the $k$-th eigenvalue on $M$, $\lambda_\eps$ and since $\varphi_i$ are orthogonal functions that
$$ \lambda_\eps \leq \max_{i=1,\cdots, k+1} \frac{\int_M \left\vert \nabla \varphi_i \right\vert^2_g Q_\eps \lambda_\eps^{\frac{n-2}{2}} dv_g}{\int_M \left(\varphi_i\right)^2 e^{nu_\eps} dv_g} \leq \lambda_\eps^{\frac{n}{2}} \lambda_\eps^{\frac{n-2}{2}} = \lambda_\eps . $$
The case of equality in the min-max characterization of $\lambda_\eps$ implies that a linear combination of $\varphi_i$ is an eigenfunction for $\lambda_\eps$ which is impossible since it vanishes on an open set.
\end{proof}

\subsubsection{Non concentration of the energy}
Far from bad points, the energy cannot have concentration points:
\begin{cl}
Let $p\in M \setminus \{p_1,\cdots, p_s\}$, then
\begin{equation} \label{noconcentdimn} \begin{split}
\lim_{r\to 0} \limsup_{\eps\to 0} \int_{\mathbb{B}_r(p)}  e^{nu_{\eps}}   & = \lim_{r\to 0} \limsup_{\eps\to 0}\int_{\mathbb{B}_r(p)} Q_\eps \left\vert \nabla \frac{\Phi_{\eps}}{\omega_{\eps}} \right\vert^2 \\
 & = \lim_{r\to 0} \limsup_{\eps\to 0}\int_{\mathbb{B}_r(p)} Q_\eps \left\vert \nabla \Phi_{\eps} \right\vert^2 =  0 \end{split}
\end{equation}
\end{cl}

\begin{proof} Let $\eta \in \mathcal{C}_c^{\infty}(\mathbb{D}_{\sqrt{r}}(p))$ with $0\leq \eta \leq 1$, $\eta = 1$ in $\mathbb{B}_{r}(p)$ and $\int_{M}\left\vert \nabla \eta\right\vert^n_{g} \leq \frac{C}{\ln\left(\frac{1}{r}\right)}$, and we first have using \ref{clbadpoints}
\begin{equation*} \begin{split} \int_{M} \eta e^{nu_{\eps}} \leq \left( \int_{M} \eta^2 e^{nu_{\eps}}\right)^{\frac{1}{2}} & \leq \left(\frac{1}{\lambda_{\star}\left(\mathbb{B}_{\sqrt{r}}(p),e^{nu_{\eps}} , Q_\eps\right)} \int_{\mathbb{B}_{\sqrt{r}}(p)} \left\vert \nabla \eta \right\vert^2 \right)^{\frac{1}{2}} \\
& \leq  \left( \frac{1}{\lambda_{\eps}^{\frac{2}{n}}} \left( \int_{\mathbb{B}_{\sqrt{r}}(p)} \left\vert \nabla \eta \right\vert^n \right)^{\frac{2}{n}} \right)^{\frac{1}{2}}    \leq \left( \frac{C}{ \lambda_\eps \ln\frac{1}{r} }\right)^{\frac{1}{n}}  
\end{split} \end{equation*}
Now, we integrate the equation $- div_g \left(Q_\eps \nabla \Phi_{\eps} \right) = \lambda_{\eps}^{\frac{n}{2}} e^{nu_{\eps}} \Phi_{\eps}$ against $\eta  \frac{\Phi_{\eps}}{\omega_{\eps}^2}$ and we obtain
$$ \int_{M}Q_\eps \eta \left\langle \nabla{ \frac{ \Phi_{\eps}}{\omega_{\eps}^2}},\nabla\Phi_{\eps} \right\rangle + \int_{M}Q_\eps \nabla \eta \nabla{ \Phi_{\eps}} \frac{\Phi_{\eps}}{\omega_\eps^2}  = \lambda_\eps^{\frac{n}{2}} \int_{M} \eta  e^{nu_{\eps}} $$
so that
\begin{equation*}
\begin{split}
  \int_{M} \eta Q_\eps \left\vert \nabla \frac{\Phi_\eps}{\omega_\eps} \right\vert^2  = & \lambda_\eps^{\frac{n}{2}} \int_{M}  \eta e^{nu_{\eps}}  - \int_{M} \eta Q_\eps  \nabla\frac{\Phi_{\eps}}{\omega_{\eps}^2}\nabla \omega_{\eps}  \frac{\Phi_{\eps}}{\omega_{\eps}} \\ 
  & - \int_{M} Q_\eps \nabla \eta \nabla{\Phi_{\eps}} \frac{\Phi_{\eps}}{\omega_\eps^2} \\
  \leq & \left( \frac{C}{ \ln\frac{1}{r}} \right)^{\frac{1}{n}} +  \int_M \eta Q_\eps  \frac{\left\vert \nabla \omega_\eps \right\vert^2 }{\omega_\eps} \\
  & + \left(  \int_{\mathbb{B}_{\sqrt{r}}(p)}  Q_\eps \frac{\left\vert \nabla{\omega_\eps} \right\vert^2}{\omega_\eps^2} \right)^{\frac{1}{2}}  \left(  \int_{M} Q_\eps  \left\vert \nabla\eta \right\vert^2 \right)^{\frac{1}{2}} 
 \end{split}
 \end{equation*}
where we used that $\left\vert \Phi_\eps \right\vert =\omega_\eps$, so that 
$$  \int_{ \mathbb{B}_{\sqrt{r}}(p) } \eta Q_\eps  \left\vert \nabla \frac{\Phi_\eps}{\omega_\eps} \right\vert^2  \leq \left( \frac{C}{\ln\frac{1}{r}} \right)^{\frac{1}{n}} \left(1+O\left(\delta_\eps^{\frac{1}{2}}\right)\right)  +O\left(\delta_{\eps} \right) $$
Now, to conclude, we have by \eqref{eqradialreplacementdimn} that for any $p\in M$,
$$ \int_{ \mathbb{B}_{\sqrt{r}}(p) } Q_\eps \left\vert \nabla \Phi_{\eps} \right\vert^2 - \int_{ \mathbb{B}_r(p) } Q_\eps \left\vert \nabla \frac{\Phi_{\eps}}{\omega_{\eps}} \right\vert^2 \leq O\left( \delta_{\eps}^{\frac{1}{2}} \right) $$
and the proof of the Claim is complete.
\end{proof}

\subsubsection{$W^{1,n}$-convergence of eigenfunctions and $n$-harmonic replacement}

Let $p\in M\setminus\{p_1,\cdots,p_s\}$, and $\rho_\eps > 0$ be such that $\mathbb{B}_{\rho_\eps}(p) \subset M\setminus\{p_1,\cdots,p_s\}$ . We denote by $\Psi_\eps : \mathbb{B}_{\rho_\eps}(p) \to \mathbb{S}^{p_\eps}$ the $n$-harmonic replacement of $\frac{\Phi_\eps}{\omega_\eps}$ on $\mathbb{B}_{\rho_\eps}(p)$ and we set 
$$ P_\eps = \left( \left\vert \nabla \Psi_\eps \right\vert^{2}  \right)^{\frac{n-2}{2}}.$$

\begin{cl} The four following upper estimates occur
\begin{equation} \label{estdiffphioveromegapsidimn} \int_{\mathbb{B}_{\rho_\eps}(p)}  Q_\eps \left\vert \nabla  \frac{\Phi_\eps}{\omega_\eps}  \right\vert^2_g - \int_{\mathbb{B}_{\rho_\eps}(p)}  Q_\eps \left\vert \nabla  \Psi_{\eps} \right\vert^2_g \leq O\left(\delta_\eps^{\frac{1}{2}}\right)  \end{equation}
\begin{equation} \label{estdiffphipsidimn} \int_{\mathbb{B}_{\rho_\eps}(p)}  Q_\eps \left\vert \nabla \Phi_\eps  \right\vert^2_g - \int_{\mathbb{B}_{\rho_\eps}(p)}  Q_\eps \left\vert \nabla  \Psi_{\eps} \right\vert^2_g \leq O\left(\delta_\eps^{\frac{1}{2}}\right)  \end{equation}
\begin{equation} \label{estdiffphipsidimnP} \int_{\mathbb{B}_{\rho_\eps}(p)}  P_\eps \left\vert \nabla \Phi_\eps \right\vert^2_g - \int_{\mathbb{B}_{\rho_\eps}(p)}  P_\eps \left\vert \nabla  \Psi_{\eps} \right\vert^2_g \leq O\left(\delta_\eps^{\frac{1}{2}} + \delta_\eps^{\frac{2}{n}}\right)  \end{equation}
\begin{equation} \label{estdiffphioveromegapsidimnP} \int_{\mathbb{B}_{\rho_\eps}(p)}  P_\eps \left\vert \nabla  \frac{\Phi_\eps}{\omega_\eps}  \right\vert^2_g - \int_{\mathbb{B}_{\rho_\eps}(p)}  P_\eps \left\vert \nabla  \Psi_{\eps} \right\vert^2_g \leq O\left(\delta_\eps^{\frac{1}{2}} + \delta_\eps^{\frac{2}{n}}\right)  \end{equation}
as $\eps \to 0$.
\end{cl}

\begin{proof}
We first prove \eqref{estdiffphioveromegapsidimn}
We test the function $\frac{\Phi_{\eps}^i}{\omega_{\eps}}- \Psi_{\eps}^i$ in the variational characterization of $\lambda_{\star}:= \lambda_{\star}\left(\mathbb{B}_{\rho_\eps}(p),e^{nu_{\eps}} , \frac{Q_\eps}{\lambda_\eps^{\frac{n-2}{2}}} \right)$:
$$ \lambda_{\eps}^{\frac{n}{2}} \int_{ \mathbb{B}_{\rho_\eps}(p) } \left(\frac{\Phi_{\eps}^i}{\omega_{\eps}}- \Psi_{\eps}^i\right)^2e^{nu_{\eps}}  \leq \frac{\lambda_{\eps}}{\lambda_\star} \int_{ \mathbb{B}_{\rho_\eps}(p) } Q_\eps \left\vert \nabla\left(\frac{\Phi_{\eps}^i}{\omega_{\eps}}- \Psi_{\eps}^i\right)\right\vert^2$$
and we sum on $i$ to get thanks to Claim \ref{clbadpoints}
\begin{equation} \label{eqtestlambdastardimn} \begin{split} \lambda_{\eps}^{\frac{n}{2}} & \int_{ \mathbb{B}_{\rho_\eps}(p) } \left\vert \frac{\Phi_{\eps}}{\omega_{\eps}}- \Psi_{\eps}\right\vert^2e^{nu_{\eps}} \\ \leq &  \int_{ \mathbb{B}_{\rho_\eps}(p) } Q_\eps \left\vert \nabla\frac{\Phi_{\eps}}{\omega_{\eps}}\right\vert^2 +  \int_{ \mathbb{B}_{\rho_\eps}(p) }Q_\eps\left\vert \nabla \Psi_{\eps} \right\vert^2 - 2\int_{ \mathbb{B}_{\rho_\eps}(p) }Q_\eps \nabla\frac{\Phi_{\eps}}{\omega_{\eps}}\nabla \Psi_{\eps} \end{split} \end{equation}
Now we test the equation $ - div_g\left(\lambda_{\eps}^{\frac{2-n}{2}} Q_\eps \Phi_{\eps} \right) =  \lambda_{\eps} e^{nu_{\eps}} \Phi_{\eps} $ against $\frac{\Phi_{\eps}}{\omega_{\eps}^2} - \frac{\Psi_{\eps}}{\omega_{\eps}}$ and we get after an integration by part knowing that $\frac{\Phi_{\eps}}{\omega_{\eps}^2} - \frac{\Psi_{\eps}}{\omega_{\eps}} = 0$ on $\partial \mathbb{B}_r(p)$
\begin{equation} \begin{split} \int_{\mathbb{B}_{\rho_\eps}(p)} &Q_\eps \left(\frac{1}{\omega_{\eps}}\nabla \Phi_{\eps} \nabla\left(\frac{\Phi_{\eps}}{\omega_{\eps}} - \Psi_{\eps}\right) + \nabla\frac{1}{\omega_{\eps}}\nabla \Phi_{\eps}  \left(\frac{\Phi_{\eps}}{\omega_{\eps}} - \Psi_{\eps}\right) \right) \\
= & \lambda_\eps^{\frac{n}{2}} \int_{\mathbb{B}_{\rho_\eps}(p)} \left\langle \frac{\Phi_{\eps}}{\omega_{\eps}}, \frac{\Phi_{\eps}}{\omega_{\eps}} - \Psi_{\eps} \right\rangle e^{nu_{\eps}} \end{split} \end{equation}
so that,
\begin{equation} \label{eqtestagainstequationdimn}
\begin{split}
 \int_{\mathbb{B}_{\rho_\eps}(p)} & Q_\eps \nabla \frac{\Phi_{\eps}}{\omega_{\eps}} \nabla\left(\frac{\Phi_{\eps}}{\omega_{\eps}} - \Psi_{\eps}\right) = \lambda_\eps^{\frac{n}{2}} \int_{\mathbb{B}_{\rho_\eps}(p)} \left\langle \frac{\Phi_{\eps}}{\omega_{\eps}}, \frac{\Phi_{\eps}}{\omega_{\eps}} - \Psi_{\eps} \right\rangle e^{nu_{\eps}} \\ & -  \int_{\mathbb{B}_{\rho_\eps}(p)} Q_\eps \nabla\frac{1}{\omega_{\eps}}\nabla \Phi_{\eps}  \left(\frac{\Phi_{\eps}}{\omega_{\eps}} - \Psi_{\eps}\right) 
+ \int_{\mathbb{B}_{\rho_\eps}(p)} Q_\eps \Phi_{\eps}  \nabla \frac{1}{\omega_{\eps}} \nabla\left(\frac{\Phi_{\eps}}{\omega_{\eps}} - \Psi_{\eps}\right) 
\end{split}
 \end{equation}
 Knowing that $\left\vert \frac{\Phi_{\eps}}{\omega_{\eps}} \right\vert = \left\vert \Psi_\eps \right\vert$, it is clear that $2\left\langle \frac{\Phi_{\eps}}{\omega_{\eps}}, \frac{\Phi_{\eps}}{\omega_{\eps}} - \Psi_{\eps} \right\rangle = \left\vert \frac{\Phi_{\eps}}{\omega_{\eps}}- \Psi_{\eps}\right\vert^2$ and multiplying \eqref{eqtestagainstequationdimn} by $2$, we obtain
\begin{equation}  \label{eqtestagainstequation2dimn}
\begin{split}
& 2 \int_{\mathbb{B}_{\rho_\eps}(p)} Q_\eps \left\vert \nabla \frac{\Phi_{\eps}}{\omega_{\eps}}\right\vert^2 - 2 \int_{\mathbb{B}_{\rho_\eps}(p)} Q_\eps \nabla \frac{\Phi_{\eps}}{\omega_{\eps}} \nabla \Psi_{\eps}  = \lambda_\eps^{\frac{n}{2}} \int_{\mathbb{B}_{\rho_\eps}(p)} \left\vert \frac{\Phi_{\eps}}{\omega_{\eps}}- \Psi_{\eps}\right\vert^2 e^{nu_{\eps}} 
\\ & + 2 \int_{\mathbb{B}_{\rho_\eps}(p)} Q_\eps \frac{\nabla\omega_{\eps}}{\omega_{\eps}} \left( \frac{\nabla \Phi_{\eps}}{\omega_{\eps}}  \left(\frac{\Phi_{\eps}}{\omega_{\eps}} - \Psi_{\eps}\right) - \frac{\Phi_{\eps}}{\omega_{\eps}} \nabla\left(\frac{\Phi_{\eps}}{\omega_{\eps}} - \Psi_{\eps} \right) \right) \\
 \leq & \lambda_\eps^{\frac{n}{2}} \int_{\mathbb{B}_{\rho_\eps}(p)} \left\vert \frac{\Phi_{\eps}}{\omega_{\eps}}- \Psi_{\eps}\right\vert^2 e^{nu_{\eps}} + O\left( \delta_{\eps}^{\frac{1}{2}}\right)
\end{split}
  \end{equation}
Summing \eqref{eqtestlambdastardimn} and \eqref{eqtestagainstequation2dimn}, we get \eqref{estdiffphioveromegapsidimn}. Then, \eqref{estdiffphipsidimn} easily follows from \eqref{estdiffphioveromegapsidimn} and \eqref{eqradialreplacementdimn} by a triangle inequality. Now let's prove \eqref{estdiffphipsidimnP}. We have that
\begin{equation*} \begin{split} \int_{\mathbb{B}_{\rho_\eps}(p)}  P_\eps \left\vert \nabla \Phi_\eps \right\vert^2_g - \int_{\mathbb{B}_{\rho_\eps}(p)} & P_\eps  \left\vert \nabla \Psi_\eps \right\vert_g^{2} =    \int_{\mathbb{B}_{\rho_\eps}(p)}  Q_\eps \left\vert \nabla \Phi_\eps \right\vert^2_g - \int_{\mathbb{B}_{\rho_\eps}(p)} Q_\eps  \left\vert \nabla \Psi_\eps \right\vert_g^{2} \\
&  + \int_{\mathbb{B}_{\rho_\eps}(p)} \left(\left\vert \nabla \Psi_\eps\right\vert^{n-2} - \left\vert \nabla \Phi_\eps\right\vert^{n-2}\right) \left( \left\vert \nabla \Phi_\eps \right\vert^2_g - \left\vert \nabla\Psi_\eps \right\vert^2_g \right) \\
&  + \int_{\mathbb{B}_{\rho_\eps}(p)} \left(\left\vert \nabla \Phi_\eps\right\vert^{n-2} - Q_\eps \right) \left( \left\vert \nabla \Phi_\eps \right\vert^2_g - \left\vert \nabla\Psi_\eps \right\vert^2_g \right) \\
\leq & O\left( \delta_\eps^{\frac{1}{2}}\right) + O\left( \delta_\eps^{\frac{2}{n}}  \right)
\end{split} \end{equation*}
where we used \eqref{estdiffphipsidimn} and \eqref{eqQepsPhiepsdimn}. Finally, we prove \eqref{estdiffphioveromegapsidimnP}. We have by \eqref{estdiffphipsidimnP} that
\begin{equation*} \begin{split} \int_{\mathbb{B}_{\rho_\eps}(p)}  P_\eps \left\vert \nabla \frac{\Phi_{\eps}}{\omega_{\eps}} \right\vert^2_g - \int_{\mathbb{B}_{\rho_\eps}(p)} P_\eps  \left\vert \nabla \Psi_\eps \right\vert_g^{2} \leq &  O\left(\delta_\eps^{\frac{1}{2}} +  \delta_\eps^{\frac{2}{n}}  \right) + \int_{\mathbb{B}_{\rho_\eps}(p)} P_\eps \left( \left\vert \nabla \frac{\Phi_{\eps}}{\omega_{\eps}} \right\vert^2_g - \left\vert \nabla\Phi_\eps \right\vert^2_g \right) \\
\leq & O\left( \delta_\eps^{\frac{1}{2}} +  \delta_\eps^{\frac{2}{n}}\right) + \int_{\mathbb{B}_{\rho_\eps}(p)} P_\eps \left(1-\omega_\eps^2\right) \left\vert \nabla \frac{\Phi_{\eps}}{\omega_{\eps}} \right\vert^2_g \\
\leq & O\left( \delta_\eps^{\frac{1}{2}} +  \delta_\eps^{\frac{2}{n}} + \delta_\eps\right)
\end{split} \end{equation*}
since $\omega_\eps^2 \geq 1-\delta_\eps$ and we obtain \eqref{estdiffphioveromegapsidimnP}. The proof is complete
\end{proof}

\begin{cl} \label{clmain} There is $\gamma_{n,g} >0$ such that if
$$\int_{\mathbb{B}_{\rho}(p)} Q_{\eps}^{\frac{n}{n-2}}dv_g \leq \gamma_{n,g} $$
then, there is $\frac{\rho}{2} <\rho_\eps < \rho $ such that the harmonic replacement $\Psi_{\eps} : \mathbb{B}_{\rho_\eps}(p) \to \mathbb{S}^{p_\eps}$ of $\frac{\Phi_\eps}{\omega_\eps}$ satisfies for any $r < \frac{{\rho}}{4}$:
\begin{equation} \label{mainestdimn} \int_{\mathbb{B}_r(p)}  \left(P_\eps + Q_\eps\right) \left\vert \nabla \left( \Phi_\eps - \Psi_{\eps}\right) \right\vert^2_g \to 0  \end{equation}
as $\eps \to 0$.

\end{cl}

\begin{proof} 
\textbf{STEP 1:} Up to reduce $\gamma_{n,g}>0$, there is $\frac{\rho}{2}<\rho_\eps < \rho$ such that up to a rotation coordinates of $\Psi_\eps$, we have that the first coordinate of the $n$-harmonic replacement $\Psi_\eps : \mathbb{B}_{\rho_\eps}(p) \to \mathbb{S}^{p_\eps}$ is uniformly lower bounded
\begin{equation} \label{eqphiepsuniformlylowerbounded}
\forall x \in \mathbb{B}^{n}_{\rho_\eps}(p) , \psi_\eps^{1}(x ) \geq \frac{1}{2}. \end{equation}

\textbf{Proof of Step 1:} We apply the Courant-Lebesgue lemma to obtain $\frac{\rho}{2} < \rho_\eps < \rho$ such that
\begin{equation*} 
\begin{split} \int_{\mathbb{S}^{n-1}}\left( \sum_{i=1}^{p_\eps+1}\left\vert \partial_\theta \psi_\eps^{i} \left( \rho_\eps . \theta \right) \right\vert^{2} \right)^{\frac{n}{2}} d\theta_{\mathbb{S}^n-1} \leq & \frac{1}{\ln 2} \int_{\mathbb{B}_{\rho}(p)} \left( \sum_{i=1}^{p_\eps+1} \left\vert \nabla \psi_\eps^{i} \left( y \right) \right\vert^{2} \right)^{\frac{2}{n}} dy \\
\leq & \frac{C_g}{\ln 2} \left( \int_{\mathbb{B}_{\rho}(p)} \left\vert \nabla \Psi_\eps \right\vert^{n}_g dv_g\right)^{\frac{2}{n}} 
\end{split}
\end{equation*}
for a constant $C_g$ only depending on $g$. Now, for a given $1\leq i_0\leq p_\eps+1$, we have by classical Morrey-Sobolev injections a constant $a_n$ such that for any $z,z' \in \mathbb{S}^{n-1}$
\begin{equation*} 
\begin{split}
\left\vert \phi_\eps^{i_0}(\rho_\eps z+p) - \phi_\eps^{i_0}(\rho_\eps z'+p) \right\vert \leq & a_n \left( \int_{\mathbb{S}^{n-1}}\left\vert \partial_\theta \phi_\eps^{i_0} \left( \rho_\eps\theta \right) \right\vert^{n}  d\theta_{\mathbb{S}^n-1} \right)^{\frac{1}{n}}  d_{\mathbb{S}^n-1}(z,z')^{\frac{1}{n}} \\
\leq & a_n 2^{\frac{1}{n}} \frac{C_g}{\ln 2} \left( \int_{\mathbb{B}_{\rho}(p)} \left\vert \nabla \Phi_\eps \right\vert^{n}_g dv_g\right)^{\frac{1}{n}} \\
\leq & \frac{1}{8}
\end{split}
 \end{equation*}
 where we choose $a_n 2^{\frac{1}{n}} \frac{C_g}{\ln 2} \gamma_{n,g}^{\frac{1}{n}} \leq \frac{1}{8}$. Now, let $z_0^\eps \in  \mathbb{S}^{n-1}$ be such that 
 $$\omega_\eps(\rho_\eps z_0^\eps + p) = \max_{z\in  \mathbb{S}^{n-1}} \omega_\eps(\rho_\eps z + p)$$
 Then, up to a rotation of $\Phi_\eps$, we assume that $\phi_\eps^{1}(\rho_\eps z_0^\eps+p) = \omega_\eps(\rho_\eps z_0^\eps+p)$ and $\phi_\eps^{i}(\rho_\eps z_0+p) = 0$ for $i \neq 1$. Knowing in addition that $\omega_\eps^2 \geq 1-\delta_\eps$, this implies that 
\begin{equation} \label{eqlowerboundpsieps1boundary}
\forall z \in \mathbb{S}^{n-1} , \frac{\phi_{\eps}^{1}(\rho_\eps z+p)}{\omega_\eps(\rho_\eps z+p)} \geq \frac{\phi_\eps^{1}(\rho_\eps z_0^\eps+p)-\frac{1}{8 }}{\omega_\eps(\rho_\eps z_0^\eps+p)} \geq 1 - \frac{1}{8\sqrt{1-\delta_\eps}} \geq \frac{3}{4} . \end{equation}
Now, let's focus on the $n$-harmonic extension $\Psi_\eps$ of $\frac{\Phi_\eps}{\omega_\eps}$ on $\mathbb{B}_{\rho_\eps}(p)$. We define the following extension of $\Psi_\eps$ in $\mathbb{R}^n$ by setting $\widetilde{\Psi_\eps}(\rho_\eps z) = \Psi_\eps\left(\rho_\eps \frac{z}{\left\vert z \right\vert^2}\right)$ if $\left\vert z \right\vert\geq 1$ so that we have by conformal invariance of the $n$-energy
$$ \int_{\mathbb{R}^n} \left\vert \nabla \widetilde{\Psi_\eps} \right\vert^n_g dv_g \leq  2c_g  \int_{\mathbb{B}_{\rho_\eps}(p)} \left\vert \nabla \Psi_\eps \right\vert^n_gdv_g \leq 2 c_g \gamma_{n,g} $$
Now let $x_\eps$ be such that $\psi_\eps^1(x_\eps) = \min_{x\in \overline{\mathbb{B}_{\rho_\eps}(p)}} \psi_\eps^1(x)$, we aim at proving that $\psi_\eps^1(x_\eps) \geq \frac{1}{2}$. We assume that $\left\vert x_\eps \right\vert < \rho_\eps $ since otherwise, \eqref{eqlowerboundpsieps1boundary} proves step 1. We set $d_\eps = \rho_\eps - \left\vert x_\eps \right\vert$ and $y_\eps = \rho_\eps \frac{x_\eps}{\left\vert x_\eps \right\vert}$. By application of the Courant-Lebesgue theorem and Morrey-Sobolev embeddings again (see the beginning of Step 1), we have existence of a radius $\frac{d_\eps}{2} < R_\eps < d_\eps $ such that
$$ \sup_{w,w' \in  \mathbb{S}_{R_\eps}^{n-1}(y_\eps)} \left\vert \psi_\eps^{1}(w) - \psi_\eps^{1}(w') \right\vert \leq a_n 2^{\frac{1}{n}}\frac{C_g}{\ln 2}  \left(\int_{\mathbb{B}_{\rho_\eps}(p)} \left\vert \nabla \widetilde{\Psi_\eps} \right\vert^n_gdv_g\right)^{\frac{1}{n}}  \leq \frac{1}{8} $$
if we assume $a_n 2^{\frac{1}{n}}\frac{C_g}{\ln 2} 2 c_g \gamma_{n,g}^{\frac{1}{n}} \leq \frac{1}{8}$. Now the $\eps$-regularity result independant from the dimension of the target manifold (see Claim \ref{clindependanttargetmanifoldepsreg}) gives that
$$ \sup_{w \in \mathbb{B}^n_{\frac{d_\eps}{2}}(x_\eps)} \left\vert \psi_\eps^1(w) - \psi_\eps^1(x_\eps) \right\vert \leq C_{n,g}^{\frac{1}{n}} \left(\int_{\mathbb{B}^n_{d_\eps}(x_\eps)} \left\vert \nabla \Psi_\eps \right\vert^n_g\right)^{\frac{1}{n}}  \leq \frac{1}{8}$$
if we assume that $C_{n,g}^{\frac{1}{n}} \gamma_{n,g}^{\frac{1}{n}} \leq \frac{1}{8}$. Then, gathering all the previous inequalities and noticing that $\mathbb{B}^n_{\frac{d_\eps}{2}}(x_\eps) \cap \mathbb{S}_{R_\eps}(y_\eps) \neq \emptyset$ and that $\mathbb{S}_{\rho_\eps}(p) \cap  \mathbb{S}_{R_\eps}(y_\eps) \neq \emptyset$ we obtain the desired lower bound.

\medskip

\textbf{STEP 2:} We prove that
\begin{equation} \label{eqstep2convergesto0} 
\begin{split} \sum_i \int_{\mathbb{B}_{\rho_\eps}(p)} P_\eps \left\vert  \nabla \left(\frac{\phi_\eps^i}{\omega_\eps}-\psi_\eps^i\right) - \nabla\left( \ln \psi_\eps^{1} \right) \left(\frac{\phi_\eps^i}{\omega_\eps}-\psi_\eps^i\right) \right\vert^2_g dv_g \\ = \int_{\mathbb{B}_{\rho_\eps}(p)} P_\eps\left( \left\vert \nabla \frac{\Phi_\eps}{\omega_\eps} \right\vert^2 - \left\vert \nabla \Psi_\eps \right\vert^2 \right) \leq O\left(\delta_\eps^{\frac{1}{2}}+ \delta_\eps^{\frac{2}{n}}\right)
\end{split}\end{equation}

\textbf{Proof of Step 2:}

We recall that $\Psi_{\eps}:\mathbb{B}_{\rho_\eps}(p) \to \mathbb{S}^{p_\eps} $ is the $n$-harmonic replacement of $\frac{\Phi_{\eps}}{\omega_{\eps}}$ and we obtain
\begin{equation*} 
\begin{split}
\int_{\mathbb{B}_{\rho_\eps}(p)}  P_\eps \left\vert \nabla \left(\frac{\Phi_{\eps}}{\omega_{\eps}} - \Psi_{\eps}\right) \right\vert^2_g  & - \left(\int_{\mathbb{B}_{\rho_\eps}(p)} P_\eps \left\vert \nabla \frac{\Phi_{\eps}}{\omega_{\eps}} \right\vert^2_g - \int_{\mathbb{B}_{\rho_\eps}(p)} P_\eps \left\vert \nabla \Psi_{\eps} \right\vert^2_g \right)  \\ = & 2 \int_{\mathbb{B}_{\rho_\eps}(p)} P_\eps \left\langle \nabla \Psi_\eps, \nabla \left(\Psi_\eps - \frac{\Phi_{\eps}}{\omega_{\eps}}\right) \right\rangle_g \\
= &    2 \int_{\mathbb{B}_{\rho_\eps}(p)}  -div_g\left(P_\eps \nabla \Psi_\eps\right) \left(\Psi_\eps - \frac{\Phi_{\eps}}{\omega_{\eps}}\right)  \\
= &  2 \int_{\mathbb{B}_{\rho_\eps}(p)}  P_\eps \left\vert \nabla \Psi_\eps \right\vert_g^2 \Psi_\eps \left(\Psi_\eps - \frac{\Phi_{\eps}}{\omega_{\eps}}\right) \\
= &   \int_{\mathbb{B}_{\rho_\eps}(p)}  P_\eps \left\vert \nabla \Psi_\eps \right\vert_g^2  \left\vert\Psi_\eps - \frac{\Phi_{\eps}}{\omega_{\eps}}\right\vert^2
 \end{split}
\end{equation*}
Now, for a given $u \in \mathcal{C}^{\infty}\left(\mathbb{B}_{\rho_\eps}(p)\right) \cap \mathcal{C}^{0}\left(\overline{\mathbb{B}_{\rho_\eps}(p)}\right)$, such that $u=0$ we have the following computation:
\begin{equation*} 
\begin{split}
 \int_{\mathbb{B}_{\rho_\eps}(p)} P_\eps \left\vert  \nabla u - \nabla\left( \ln \psi_\eps^{1} \right) u \right\vert^2_g dv_g  = & \int_{\mathbb{B}_r(p)} P_\eps \left\vert \nabla u \right\vert_g^2 + \int_{\mathbb{B}_{\rho_\eps}(p)} P_\eps \frac{\left\vert \nabla \psi_\eps^{1} \right\vert_g^2}{\left(\psi_\eps^{1}\right)^2} u^2 \\
 & - \int_{\mathbb{B}_{\rho_\eps}(p)} P_\eps \frac{\nabla \psi_\eps^{1}}{\psi_\eps^{1}} \nabla u^2 \\
  = & \int_{\mathbb{B}_{\rho_\eps}(p)} P_\eps \left\vert \nabla u \right\vert_g^2 - \int_{\mathbb{B}_{\rho_\eps}(p)} \frac{-div_g\left( P_\eps \nabla \psi_\eps^{1} \right)}{\psi_\eps^{1}}  u^2 \\
  = & \int_{\mathbb{B}_{\rho_\eps}(p)} P_\eps \left\vert \nabla u \right\vert_g^2 - \int_{\mathbb{B}_{\rho_\eps}(p)} P_\eps \left\vert \nabla \Psi_{\eps} \right\vert_g^2  u^2
 \end{split}
\end{equation*}
so that testing this for $u= \frac{\phi_\eps^i}{\omega_\eps}-\psi_\eps^i$ and summing over $i$, we obtain
\begin{equation*} 
\begin{split}
 \sum_i \int_{\mathbb{B}_{\rho_\eps}(p)} P_\eps \left\vert  \nabla \left(\frac{\phi_\eps^i}{\omega_\eps}-\psi_\eps^i\right) - \nabla\left( \ln \psi_\eps^{1} \right) \left(\frac{\phi_\eps^i}{\omega_\eps}-\psi_\eps^i\right) \right\vert^2_g dv_g \\
 = \int_{\mathbb{B}_{\rho_\eps}(p)} P_\eps \left\vert \nabla\left( \frac{\Phi_\eps}{\omega_\eps}-\Psi_\eps \right) \right\vert_g^2 - \int_{\mathbb{B}_{\rho_\eps}(p)} P_\eps \left\vert \nabla \Psi_{\eps} \right\vert_g^2  \left\vert\frac{\Phi_\eps}{\omega_\eps}-\Psi_\eps\right\vert^2.
  \end{split}
\end{equation*}
Finally, using the previous computations and \eqref{estdiffphioveromegapsidimnP} we obtain \eqref{eqstep2convergesto0}

\medskip

\textbf{STEP 3:} We prove that
\begin{equation} \label{eqstep3convergesto0Peps}  
\int_{\mathbb{B}_{\rho_\eps}(p)}   \left( \left\vert \nabla \Psi_\eps\right\vert^{n-2} + \left\vert \nabla \Phi_\eps\right\vert^{n-2} \right)  \left\vert\nabla \left( \frac{\Phi_\eps}{\omega_\eps}   - \Phi_\eps \right)\right\vert^2 = O\left(\delta_\eps^{\frac{2}{n}} \right) \text{ as } \eps\to 0
 \end{equation}
and that
\begin{equation} \label{eqstep3convergesto0}  
\int_{\mathbb{B}_{\rho_\eps}(p)} \left( \left\vert \nabla \Psi_\eps\right\vert^{n-2} + \left\vert \nabla \Phi_\eps\right\vert^{n-2} \right) \left( \left\vert\nabla \Phi_\eps \right\vert  - \left\vert \nabla \Psi_\eps\right\vert\right)^2 =  O\left(\delta_\eps^{\frac{1}{2}} + \delta_\eps^{\frac{2}{n}} \right) \text{ as } \eps\to 0
 \end{equation}

\textbf{Proof of Step 3:} By the same computations as for the proof of \eqref{eqradialreplacementdimn}, we have
$$ \left\vert \nabla \left(\Phi_\eps - \frac{\Phi_{\eps}}{\omega_{\eps}}\right) \right\vert^2 \leq \frac{ \left\vert \nabla\omega_\eps \right\vert^2  }{\omega_\eps^2} +  \left\vert \nabla \Phi_\eps \right\vert^2 \leq 2 \left\vert \nabla \Phi_\eps \right\vert^2 \leq 2 Q_\eps^{\frac{2}{n-2}} $$
so that writing $2 = \frac{2(n-2)}{n} + \frac{4}{n} $ 
\begin{equation*} \begin{split} \int_{\mathbb{B}_r(p)} & \left( \left\vert \nabla \Psi_\eps\right\vert^{n-2} + \left\vert \nabla \Phi_\eps\right\vert^{n-2} \right) \left\vert \nabla \left(\Phi_\eps - \frac{\Phi_{\eps}}{\omega_{\eps}}\right) \right\vert^2_g \\  
\leq & 2^{\frac{2(n-2)}{n}}  \int_{\mathbb{B}_r(p)} \left( \left\vert \nabla \Psi_\eps\right\vert^{n-2} + \left\vert \nabla \Phi_\eps\right\vert^{n-2} \right) Q_\eps^{\frac{2}{n}} \left\vert \nabla \left(\Phi_\eps - \frac{\Phi_{\eps}}{\omega_{\eps}}\right) \right\vert^{\frac{4}{n}}_g \\
 \leq &  2^{\frac{2(n-2)}{n}} \left(\int_{\mathbb{B}_r(p)} \left( \left\vert \nabla \Psi_\eps\right\vert^{n-2} + \left\vert \nabla \Phi_\eps\right\vert^{n-2} \right)^{\frac{n}{n-2}} \right)^{\frac{n-2}{n}} \left(\int_{\mathbb{B}_r(p)} Q_\eps \left\vert \nabla \left(\Phi_\eps - \frac{\Phi_{\eps}}{\omega_{\eps}}\right) \right\vert^2_g\right)^{\frac{2}{n}} \\
 \leq & O\left( \delta_\eps^{\frac{2}{n}}\right)
 \end{split}  \end{equation*}
where we used \eqref{eqradialreplacementdimn} for the last inequality, and we obtain
\eqref{eqstep3convergesto0Peps}  .

Now let's prove \eqref{eqstep3convergesto0}. By \eqref{estdiffphioveromegapsidimn}, we have that
\begin{equation*} \begin{split} O\left(\delta_\eps^{\frac{1}{2}}\right)& \geq  \int_{\mathbb{B}_{\rho_\eps}(p)} Q_\eps\left( \left\vert \nabla \Phi_\eps \right\vert^2 - \left\vert \nabla \Psi_\eps \right\vert^2 \right) \\
& \geq \int_{\mathbb{B}_{\rho_\eps}(p)} P_\eps\left( \left\vert \nabla \Phi_\eps \right\vert^2 - \left\vert \nabla \Psi_\eps \right\vert^2 \right) + \int_{\mathbb{B}_{\rho_\eps}(p)} \left(Q_\eps-P_\eps\right)\left( \left\vert \nabla\Phi_\eps \right\vert^2 - \left\vert \nabla \Psi_\eps \right\vert^2 \right) \\
& \geq O\left(\delta_\eps^{\frac{1}{2}}+ \delta_\eps^{\frac{2}{n}}\right) + \int_{\mathbb{B}_{\rho_\eps}(p)} \left(Q_\eps-P_\eps\right)\left( \left\vert \nabla \Phi_\eps \right\vert^2 - \left\vert \nabla \Psi_\eps \right\vert^2 \right)
\end{split}
 \end{equation*}
where we used for the last inequality \eqref{eqstep2convergesto0}, proved in step 2 and \eqref{eqstep3convergesto0Peps}. Now,
\begin{equation*} \begin{split} \int_{\mathbb{B}_{\rho_\eps}(p)}  \left(Q_\eps-P_\eps\right) &  \left( \left\vert \nabla \Phi_\eps \right\vert^2 - \left\vert \nabla \Psi_\eps \right\vert^2 \right) \geq  \int_{\mathbb{B}_{\rho_\eps}(p)} \left(Q_\eps-\left\vert \nabla \Phi_\eps \right\vert^{n-2}\right)\left( \left\vert \nabla \Phi_\eps \right\vert^2 - \left\vert \nabla \Psi_\eps \right\vert^2 \right) \\
&+  \int_{\mathbb{B}_{\rho_\eps}(p)} \left( \left\vert \nabla \Phi_\eps \right\vert^{n-2} - \left\vert \nabla \Psi_\eps \right\vert^{n-2} \right)\left( \left\vert \nabla \Phi_\eps \right\vert^2 - \left\vert \nabla \Psi_\eps \right\vert^2 \right) \\
\geq & O\left(\delta_\eps^{\frac{2}{n}}\right) + \int_{\mathbb{B}_{\rho_\eps}(p)} \left( \left\vert \nabla \Phi_\eps \right\vert - \left\vert \nabla \Psi_\eps \right\vert \right)^2  \left( \left\vert \nabla \Phi_\eps \right\vert^{n-2} + \left\vert \nabla \Psi_\eps \right\vert^{n-2} \right)
\end{split}
 \end{equation*}
 where we used \eqref{eqQepsPhiepsdimn} again for the last inequality and we obtain  \eqref{eqstep3convergesto0} by gathering the previous inequalities.

\medskip

\textbf{STEP 4:} We prove that
\begin{equation} \label{eqstep4convergesto0}  \sum_i \int_{\mathbb{B}_{\rho_\eps}(p)} \left\vert \nabla\Phi_\eps \right\vert_g^{n-2} \left\vert  \nabla \left(\frac{\phi_\eps^i}{\omega_\eps}-\psi_\eps^i\right) - \nabla\left( \ln \psi_\eps^{1} \right) \left(\frac{\phi_\eps^i}{\omega_\eps}-\psi_\eps^i\right) \right\vert^2_g dv_g \to 0 \text{ as } \eps\to 0 \end{equation}

\textbf{Proof of Step 4:} \eqref{eqstep4convergesto0} follows from  \eqref{eqstep2convergesto0} if we prove that
$$ \sum_i \int_{\mathbb{B}_{\rho_\eps}(p)}\left( \left\vert \nabla\Phi_\eps \right\vert_g^{n-2} -  \left\vert \nabla\Psi_\eps \right\vert_g^{n-2} \right)\left\vert  \nabla \left(\frac{\phi_\eps^i}{\omega_\eps}-\psi_\eps^i\right) - \nabla\left( \ln \psi_\eps^{1} \right) \left(\frac{\phi_\eps^i}{\omega_\eps}-\psi_\eps^i\right) \right\vert^2_g dv_g \to 0 $$
as $\eps \to 0$. We denote $D_\eps$ this quantity. Using \eqref{eqstep3convergesto0} in the last inequality,  we have that 
\begin{equation*} 
\begin{split}D_\eps \leq & c_n \int_{\mathbb{B}_{\rho_\eps}(p)}\left\vert \left\vert \nabla\Phi_\eps \right\vert -  \left\vert \nabla\Psi_\eps \right\vert \right\vert \left(  \left\vert \nabla\Phi_\eps \right\vert_g^{n-3} +  \left\vert \nabla\Psi_\eps \right\vert^{n-3} \right)\left(  2 \left\vert \nabla \left( \frac{\Phi_\eps}{\omega_\eps} - \Psi_\eps \right) \right\vert^2 + 8 \left\vert \nabla \Psi_\eps \right\vert^2 \right) \\
\leq & O\left( \left(\int_{\mathbb{B}_{\rho_\eps}(p)} \left( \left\vert \nabla \Phi_\eps \right\vert - \left\vert \nabla \Psi_\eps \right\vert \right)^2  \left( \left\vert \nabla \Phi_\eps \right\vert^{n-2} + \left\vert \nabla \Psi_\eps \right\vert^{n-2} \right)\right)^{\frac{1}{2}} \right) \to 0 \text{ as } \eps\to 0
\end{split}\end{equation*}

\medskip

\textbf{STEP 5:} We prove that for any $r < \frac{\rho}{4}$
\begin{equation} \label{eqestPepsphioveromegadimn} \int_{\mathbb{B}_{r}(p)}  P_\eps \left\vert \nabla \left( \frac{\Phi_\eps}{\omega_\eps} - \Psi_{\eps}\right) \right\vert^2_g \to 0 \text{ as } \eps\to 0
\end{equation}

\textbf{Proof of Step 5:} We set $\rho_{0} = \lim_{\eps\to 0} \rho_\eps$.  From \eqref{eqstep2convergesto0}, we have that for any $\mathbb{B}_r(p) \subset \mathbb{B}_{\rho_\eps}(p)$ with $r < \rho_0$,
\begin{equation} \label{eqineqforanysubsetA}
\begin{split} \int_{\mathbb{B}_r(p)} P_{\eps}\left\vert \nabla\left(\frac{ \Phi_\eps}{\omega_\eps}-\Psi_\eps\right) \right\vert^2 \leq & 2 \int_{A} P_{\eps} \left\vert \nabla\left( \ln\psi_\eps^{1} \right) \right\vert^2  \left\vert \frac{ \Phi_\eps}{\omega_\eps}-\Psi_\eps\right\vert^2 \\
& +  2 \sum_i \int_{\mathbb{B}_r(p)} P_\eps \left\vert  \nabla \left(\frac{\phi_\eps^i}{\omega_\eps}-\psi_\eps^i\right) - \nabla\left( \ln \psi_\eps^{1} \right) \left(\frac{\phi_\eps^i}{\omega_\eps}-\psi_\eps^i\right) \right\vert^2_g dv_g  \\
\leq &2 \int_{\mathbb{B}_r(p)} P_{\eps} \left\vert \nabla\left( \ln\psi_\eps^{1} \right) \right\vert^2  \left\vert \frac{ \Phi_\eps}{\omega_\eps}-\Psi_\eps\right\vert^2 + O\left(\delta_\eps^{\frac{1}{2}}\right) \\
  \end{split}
\end{equation}
We aim at proving that the right-hand term converges to $0$ as $\eps\to 0$. 

By Claim \eqref{clW1qforPhieps} applied to $\frac{\Phi_\eps}{\omega_\eps}$ and Claim \eqref{clW1qconvergence} applied to $\Psi_\eps$, we know that up to a subsequence, the sequence of maps $\frac{\Phi_\eps}{\omega_\eps} : \mathbb{B}_{\rho_\eps}(p) \to \mathbb{S}^{p_\eps}$ and $\Psi_\eps :  \mathbb{B}_{\rho_\eps}(p) \to \mathbb{S}^{p_\eps}$ converge to some maps $\Phi_0 : \mathbb{B}_{\rho_0}(p) \to \mathbb{R}^{\mathbb{N}}$ and $\Psi_0 : \mathbb{B}_{\rho_0}(p) \to \mathbb{R}^{\mathbb{N}}$ in $W^{1,q}$ for any $1 \leq q < n$, 
and $\Psi_0$ is a harmonic function into the unit sphere of $\mathbb{R}^N$. 
Moreover, we know from \eqref{eqstep2convergesto0} and \eqref{eqstep4convergesto0} that 
$$ \left(\left\vert \nabla \Psi_\eps \right\vert_g^{\frac{n-2}{2}} +  \left\vert \nabla \frac{\Phi_\eps}{\omega_\eps} \right\vert_g^{\frac{n-2}{2}}\right) \left( \nabla \Psi_\eps - \nabla\left( \frac{\Phi_\eps}{\omega_\eps}\right) - \left\vert \nabla \ln \left( \psi_\eps^1\right)  \right\vert_g^2 \left( \Psi_\eps- \frac{\Phi_\eps}{\omega_\eps}\right) \right) \to 0 $$
strongly in $L^2$ as $\eps\to 0$. Passing to the limit as $\eps\to 0$, we obtain
$$ \left( \left\vert \nabla \Psi_0 \right\vert^{\frac{n-2}{2}} + \left\vert \nabla \Phi_0 \right\vert^{\frac{n-2}{2}} \right)\left( \nabla \Psi_0 - \nabla \Phi_0 -  \nabla \ln \left( \psi_0^1\right)  \left( \Psi_0 - \Phi_0\right) \right) = 0$$
If we define 
$$ Z = \{ z\in \mathbb{B}_{\rho_0}(p) ; \left\vert \nabla \Psi_0 \right\vert^{\frac{n-2}{2}} + \left\vert \nabla \Phi_0 \right\vert^{\frac{n-2}{2}} = 0  \}, $$
We obtain that
$$ \forall z \in \mathbb{B}_{\rho_0}(p) \setminus Z , \left( \nabla \Psi_0 - \nabla \Phi_0 - \nabla \ln \left( \psi_0^1\right)  \left( \Psi_0 - \Phi_0\right)\right)(z) = 0 $$
It is clear that $\Psi_0 - \Phi_0 = 0$ in $\partial\mathbb{B}_{\rho_0} $ as limits of functions such that $\Psi_\eps - \frac{\Phi_\eps}{\omega_\eps} = 0$ in  $\partial\mathbb{B}_{\rho_\eps}$. Now by Claim \ref{clindependanttargetmanifoldepsreg} and Step 1, we know that
$$\forall z \in \mathbb{B}_{\rho_0}(p), \left\vert \nabla \ln\psi_0^1 \right\vert^2(z) \leq 4 \frac{C_{n,g}^{\frac{2}{n}} \gamma_{n,g}^{\frac{2}{n}}}{\left(1-\left\vert z \right\vert\right)^2}  $$
Therefore, we have by a classical Hardy inequality that
\begin{equation*}
\begin{split} \int_{\mathbb{B}_{\rho_0}(p) } \left\vert \nabla\left( \Phi_0 - \Psi_0\right) \right\vert^2 = \int_{\mathbb{B}_{\rho_0}(p) \setminus Z} \left\vert \nabla\left( \Phi_0 - \Psi_0\right) \right\vert^2 = & \int_{\mathbb{B}_{\rho_0}(p) \setminus Z}    \left\vert \nabla \ln\psi_0^1 \right\vert^2 \left\vert \Phi_0 - \Psi_0\right\vert^2  \\
\leq & 4 C_{n,g}^{\frac{2}{n}} \gamma_{n,g}^{\frac{2}{n}} \int_{\mathbb{B}_{\rho_0}(p) }   \frac{ \left\vert \Phi_0 - \Psi_0\right\vert^2 }{\left(\rho_0-\left\vert z \right\vert\right)^2} dv_g \\
\leq & \tilde{C}_{n,g} \gamma_{n,g}^{\frac{2}{n}} \int_{\mathbb{B}_{\rho_0}(p) } \left\vert \nabla\left( \Phi_0 - \Psi_0\right) \right\vert_g^2dv_g
\end{split}
\end{equation*}
so that assuming that $\tilde{C}_{n,g} \gamma_{n,g}^{\frac{2}{n}}\leq \frac{1}{2}$, we obtain that $\nabla \Phi_0 =_{a.e} \nabla \Psi_0$ and then that for a.e $z \in \mathbb{B}_{\rho_0}(p)\setminus Z$, $ \Phi_0(z) =  \Psi_0(z)$ and then that $\Phi_0 =_{a.e} \Psi_0$. Now, coming back to \eqref{eqineqforanysubsetA}, it is clear that the right-hand side
$$ 2 \int_{\mathbb{B}_r(p)} P_{\eps} \left\vert \nabla\left( \ln\psi_\eps^{1} \right) \right\vert^2  \left\vert \frac{ \Phi_\eps}{\omega_\eps}-\Psi_\eps\right\vert^2 \to 0 $$
as $\eps\to 0$ by uniform boundedness of $\nabla \Psi_\eps$ and $\nabla \psi_\eps^1$ on $\mathbb{B}_r(\rho)$ (see Claim \ref{clindependanttargetmanifoldepsreg} and Step 1 of the current claim) and strong convergence to $0$ of $\frac{\Phi_{\eps}}{\omega_\eps} - \Psi_\eps $ in $L^2$. The proof of \eqref{eqestPepsphioveromegadimn} is complete.

\medskip

\textbf{STEP 6:} We complete the proof of \eqref{mainestdimn}. We have:
\begin{equation*}
\begin{split}
 \int_{\mathbb{B}_r(p)}  Q_\eps \left\vert \nabla \left( \Phi_\eps - \Psi_{\eps}\right) \right\vert^2_g \leq & 2 \int_{\mathbb{B}_r(p)}  Q_\eps \left\vert \nabla \left( \Phi_\eps - \frac{\Phi_{\eps}}{\omega_\eps}\right) \right\vert^2_g + 2 \int_{\mathbb{B}_r(p)}  P_\eps \left\vert \nabla \left( \frac{\Phi_\eps}{\omega_\eps} - \Psi_{\eps}\right) \right\vert^2_g \\
& + 2 \int_{\mathbb{B}_r(p)}  \left(\left\vert \nabla \Psi_\eps \right\vert_g^{n-2} - \left\vert \nabla \Phi_\eps \right\vert_g^{n-2} \right) \left\vert \nabla \left( \frac{\Phi_\eps}{\omega_\eps} - \Psi_{\eps}\right) \right\vert^2_g \\
& + 2 \int_{\mathbb{B}_r(p)}  \left(\left\vert \nabla \Phi_\eps \right\vert_g^{n-2} -Q_\eps \right) \left\vert \nabla \left( \frac{\Phi_\eps}{\omega_\eps} - \Psi_{\eps}\right) \right\vert^2_g
 \end{split} \end{equation*}
The first left-hand term converges to $0$ because of \eqref{eqradialreplacementdimn}, the second one converges to $0$ because of \eqref{eqestPepsphioveromegadimn}. The third one converges to $0$ by the use of a H\"older inequality and  \eqref{eqstep3convergesto0} and the fourth one converges to 0 thanks to \eqref{eqQepsPhiepsdimn}.

The proof of \eqref{mainestdimn} is complete.
\end{proof}

We noticed that up to a subsequence, $\Psi_\eps$ and $\frac{\Phi_\eps}{\omega_\eps}$ respectively converge to some maps $\Psi_0 : \mathbb{B}_r(p)\to \mathbb{R}^{\mathbb{N}}$ and $\Phi_0 : \mathbb{B}_r(p) \to \mathbb{R}^{\mathbb{N}}$ in the following sense: for all $1\leq p<+\infty$, and $1 \leq q < n $
$$ \left(\int_{\mathbb{B}_r(p)} \left\vert \frac{\Phi_\eps}{\omega_\eps} - \Phi_0 \right\vert^p\right)^{\frac{1}{p}} +  \left(\int_{\mathbb{B}_r(p)} \left\vert \nabla\left( \frac{\Phi_\eps}{\omega_\eps} - \Phi_0 \right) \right\vert^q\right)^{\frac{1}{q}} \to 0$$
as $\eps \to 0$. $\Psi_\eps$ satisfies the same convergence properties but also more (we do not need): there is some $\alpha\in (0,1)$ such that $\left\vert \nabla \Psi_\eps- \nabla \Psi_0\right\vert$ converges to $0$ in $\mathcal{C}^{0,\alpha}$. Thanks to Claim \ref{clmain}, we obtain that
$$ \int_{\mathbb{B}_r(p)}  \left\vert \nabla \left( \frac{\Phi_\eps}{\omega_\eps} - \Psi_\eps\right) \right\vert^n \to 0 $$
so that by uniqueness of the limit, $\Psi_0 = \Phi_0$.

\subsubsection{The limiting measure has a $\mathcal{C}^{0,\alpha}$ density}
Let $\zeta \in \mathcal{C}_c^{\infty}\left( \mathbb{B}_r(p) , \mathbb{R}^{N}\right) $, we have that
\begin{equation*}
\begin{split} \int_{M} \zeta \left(\lambda_{\eps}^{\frac{n}{2}} \Phi_{\eps} e^{2u_{\eps}}dv_g -  \lambda^{\frac{n}{2}} \Phi_0 d\nu\right)  = \left(\lambda_\eps^{\frac{n}{2}} - \lambda^{\frac{n}{2}} \right) \int_{M} \zeta \Phi_\eps e^{nu_{\eps}}dv_g \\ + \lambda^{\frac{n}{2}} \left( 
   \int_{M} \zeta \left( \Phi_\eps - \Phi_0 \right) e^{nu_{\eps}}dv_g 
+  \int_{M} \zeta  \Phi_0 \left(e^{nu_{\eps}}dv_g - d\nu\right)\right)
\end{split}
\end{equation*}
Then on the first right-hand term, we have thanks to Claim \ref{clbadpoints} that
\begin{equation*} \begin{split} &\left\vert \int_{M}  \zeta \left( \Phi_{\eps} - \Phi_0\right) e^{2u_{\eps}}dA_g  \right\vert \leq  \left(\int_{\mathbb{D}_{r}(p)} \zeta^2 \left\vert \Phi_{\eps} - \Phi_0 \right\vert^2 e^{nu_{\eps}}dA_g\right)^{\frac{1}{2}} \\
& \leq  \left(\frac{\lambda_\eps^{\frac{2-n}{2}}}{\lambda_{\star}\left(\mathbb{B}_{r}(p),e^{n u_{\eps}}, \frac{ Q_\eps}{\lambda_\eps^{\frac{n-2}{2}}}\right)} \int_{\mathbb{B}_{r}(p)} Q_\eps   \left\vert\nabla \left(\zeta \left\vert \Phi_{\eps} - \Phi_0 \right\vert \right) \right\vert^2_g dv_g\right)^{\frac{1}{2}} \\
& \leq C\left( \left( \int_{\mathbb{B}_{r}(p)} Q_\eps \left\vert \nabla \left(\Phi_{\eps} - \Psi_\eps \right) \right\vert^2\right)^{\frac{1}{2}}  + \left( \int_{\mathbb{B}_{r}(p)}  \left\vert \nabla \left(\Psi_{\eps} - \Phi_0 \right) \right\vert^n \right)^{\frac{1}{n}}  \right) \\
& \to 0 \text{ as } \eps \to 0
\end{split}
\end{equation*}
for some constant $C$ independent of $\eps$. Letting $\eps\to 0$ in a weak sense to the eigenvalue equation $-div_g\left( Q_\eps \nabla \Phi_\eps \right) = \lambda_\eps^{\frac{n}{2}} \Phi_\eps e^{nu_{\eps}}$, we get
$$-div_g\left( \left\vert \nabla \Phi_0 \right\vert_g^{n-2}  \nabla \Phi_0 \right) = \lambda^{\frac{n}{2}} \Phi_0 \tilde\mu_0 $$
and since $\Phi_0$ is $n$-harmonic, we obtain that $\tilde\mu_0 = \frac{\left\vert\nabla \Phi_0 \right\vert^n}{\lambda^{\frac{n}{2}} }dv_g$ and the density is a non negative $\mathcal{C}^{0,\alpha}$ function. 

Then we obtain that for some global n-harmonic map $\Phi_0 : M\setminus \{p_1,\cdots,p_s\} \to \mathbb{R}^{\mathbb{N}}$, $\tilde\mu_0 = \frac{\left\vert\nabla \Phi_0 \right\vert^n}{\lambda^{\frac{n}{2}} }dv_g$. 
By a point removability theorem (Claim \ref{clpointremovability}), $\Phi_0$ can be extended to $\Phi_0 : M \to \mathbb{R}^{\mathbb{N}}$ as a n-harmonic map, and the conformal factor has the expected regularity. As already said, $\tilde{\mu}_i =  \frac{\left\vert\nabla \Phi_i \right\vert^n}{\lambda^{\frac{n}{2}} } dx$ for some $n$-harmonic map on the Euclidean space $\mathbb{R}^n$. A pullback by stereographic projection of the round sphere $\pi$ and a point removability theorem (Claim \ref{clpointremovability}), gives the expected regularity on $\pi^{\star} \tilde{\mu}_i$.

In the next section, we prove Theorem \ref{theodiscretespectrum}: the embedding $W^{1,2}(f . g) \to L^2(f . g)$ is compact, where $f = \frac{\left\vert \nabla \Phi_0 \right\vert_g^{2}}{\lambda} $ in $M$ (and $f = \frac{\left\vert \nabla \Phi_i \right\vert_g^{2}}{\lambda} $ in $\mathbb{R}^n$). We can deduce thanks to Remark \ref{remcompactsubdifferential} that the target sphere of $\Phi_0$ (and $\Phi_i$) can be reduced to a finite dimensional sphere. The proof of Proposition \ref{palaissmale} is complete. 

\subsection{A compact embedding for the  weighted Sobolev spaces associated to the limiting metrics}
In this section, we prove Theorem \ref{theodiscretespectrum}.
Let $f : M \to \mathbb{R}_+$ be a non-negative continuous function and we denote $Z$ its zero set. We set $L^2(f.g)$ the set of measurable functions $u : M \to \R$ such that
$$ \|  u  \|_{L^{2}(f.g)} := \int_M u^2 f^{\frac{n}{2}} dv_g < +\infty, $$

We set $ W^{1,2}(f.g) $ the completion of $ \mathcal{C}^{\infty}\left(M\right)$ with respect to the semi-norm
$$ \|  u  \|_{W^{1,2}(f.g)} = \int_M u^2 f^{\frac{n}{2}} dv_g + \int_M \left\vert \nabla u  \right\vert^2  f^{\frac{n-2}{2}} dv_g  $$
We aim at proving that for $f = \frac{\left\vert \nabla \Phi_0 \right\vert_g^{2}}{\lambda}$ given by Theorem \ref{theomain}, the embedding $W^{1,2}(f . g) \to L^2(f . g)$ is compact. For that purpose, we will prove the following local Sobolev embedding:

\begin{cl}
There is $\kappa_0 > 1$, there is $r_0>0$ and $C_0 >0$ such that for any $x \in M $ and any $u\in \mathcal{C}^{\infty}_c\left(\mathbb{B}_r(p)\right)$
$$ \left(\int_{\mathbb{B}_r(p)} u^{2\kappa_0} f^{\frac{n}{2}} dv_g \right)^{\frac{1}{\kappa_0}} \leq C_0 r^{2-\frac{n(\kappa_0-1)}{\kappa_0}}  \int_{\mathbb{B}_r(p)} \left\vert \nabla u \right\vert_g^2 f^{\frac{n-2}{2}} dv_g . $$
\end{cl}

\begin{proof}
Using a classical Sobolev inequality, we have a universal constant $C_{0}$ such that
\begin{equation*}
\begin{split} \left(\int_{\mathbb{B}_r(p)} u^{2\kappa_0} f^{\frac{n}{2}} dv_g \right)^{\frac{1}{\kappa_0}} \leq &  C_0 r^{2-\frac{n(\kappa_0-1)}{\kappa_0}}  \int_{\mathbb{B}_r(p)} \left\vert \nabla\left( u f^{\frac{n}{4\kappa_0}}\right) \right\vert_g^2  dv_g  \\
\leq & 2 C_0 r^{2-\frac{n(\kappa_0-1)}{\kappa_0}} \left(  \int_{\mathbb{B}_r(p)} \left\vert \nabla  u \right\vert_g^2 f^{\frac{n}{2\kappa_0}}  dv_g +  \int_{\mathbb{B}_r(p)} \left\vert \nabla f^{\frac{n}{4\kappa_0}} \right\vert_g^2 u^2 \right)
\end{split}
\end{equation*}
We aim at estimating the right-hand term, noticing that $f$ is nothing but the limit of $B:= \left\vert \nabla \Psi_\tau \right\vert^2_g$ as $(\tau,m) \to (0,+\infty)$, where $\Psi_\tau : \mathbb{B}_r(p)\to \mathbb{S}^m$ is a $(\tau,n)$-harmonic map, if we choose $r \leq \frac{r_\star}{2}$, where 
$$r_{\star} = \sup\left\{ r >0 , \sup_x \int_{\mathbb{B}_r(x)} \left( f + \tau\right)^{\frac{n}{2}} \leq \frac{\eps_{n,g}}{2} \right\}. $$
We recall that $\Psi_\tau$ satisfies by Claim \ref{clapproxdeltanharm}
$$  \mu^{n-2} B  \left(\kappa_g  + B\right) \geq  -   div_g\left( a^\tau \nabla \mu^n \right) +  \left\vert \nabla^2 \Psi \right\vert \mu^{n-2} +\frac{n-2}{4} \mu^{n-4} \left\vert \nabla B \right\vert^2 $$
where $\mu = \left(B + \tau\right)^{\frac{1}{2}}$ and
\begin{equation*} a_{i,j}^\tau =  \frac{1}{n}\left( \delta_{i,j} + (n-2) \frac{\nabla_i \Psi.\nabla_j \Psi}{B + \tau} \right).
 \end{equation*}
 We integrate this equation over $ u^{2} \mu^{-\nu_0}$ for $\nu_0 = \frac{n(\kappa_0 - 1)}{\kappa_0}$ and we get
\begin{equation*}
\begin{split} \int_{\mathbb{B}_r(p)} u^2 \left(a\left( \nabla \mu^{-\nu_0}\nabla\left( \mu^n\right)\right) + \frac{n-2}{4} \mu^{n-4-\nu_0} \left\vert \nabla B \right\vert^2_g   + \left\vert \nabla^2 \Psi \right\vert^2_g \mu^{n-2-\nu_0} \right)  \\
 \leq \int_{\mathbb{B}_r(p)} u^2 B \mu^{n-2-\nu_0}(B+k_g) - \int_{\mathbb{B}_r(p)} 2u \mu^{-\nu_0} a\left( \nabla u  \nabla\left( \mu^n \right)\right)
\end{split} \end{equation*}
Knowing that $\frac{1}{n} \left\vert X \right\vert^2 \leq a_{i,j}X^i X^j \leq \frac{n-1}{n}\left\vert X \right\vert^2$, and the following computations
\begin{equation*}
  \nabla \mu^{-\nu_0}\nabla\left( \mu^n \right) = - \nu_0 n \mu^{n-2-\nu_0} \left\vert \nabla \mu \right\vert^2 \text{ and } \left\vert \nabla B \right\vert^2 = 4 \mu^2 \left\vert \nabla \mu \right\vert^2
  \end{equation*}
we obtain that
\begin{equation*}
\begin{split} & \left( n-2 - (n-1)\nu_0 \right) \int_{\mathbb{B}_r(p)} u^2  \mu^{n-2-\nu_0}  \left\vert \nabla \mu \right\vert^2_g + \int_{\mathbb{B}_r(p)} u^2  \left\vert \nabla^2 \Psi \right\vert^2_g \mu^{n-2-\nu_0}  \\
 \leq & (n-1) \int_{\mathbb{B}_r(p)} 2 \left\vert u \nabla u \right\vert \mu^{n-1-\nu_0} \left\vert \nabla \mu \right\vert 
+ \int_{\mathbb{B}_r(p)} u^2 \mu^{n-\nu_0}(B+k_g) \\
 \leq & (n-1)\left( \theta^2 \int_{\mathbb{B}_r(p)} u^2  \mu^{n-2-\nu_0}  \left\vert \nabla \mu \right\vert^2_g + \frac{1}{\theta^2}  \int_{\mathbb{B}_r(p)} \mu^{n-\nu_0}  \left\vert \nabla u \right\vert^2_g\right) \\
 & + \int_{\mathbb{B}_r(p)} u^2 \mu^{n-\nu_0}(B+k_g)
\end{split} \end{equation*}
so that knowing that $ \left\vert \nabla^2 \Psi \right\vert^2_g \geq  \left\vert \nabla \mu \right\vert^2_g $ (by a H\"older inequality)
\begin{equation*}
\begin{split} \left(1 - \nu_0 - \theta^2\right) & \int_{\mathbb{B}_r(p)} u^2  \mu^{n-2-\nu_0}  \left\vert \nabla \mu \right\vert^2_g \\
\leq & \frac{1}{\theta^2}  \int_{\mathbb{B}_r(p)} \mu^{n-\nu_0}  \left\vert \nabla u \right\vert^2_g + \frac{1}{n-1} \int_{\mathbb{B}_r(p)} u^2 \mu^{n-\nu_0}(B+k_g)
\end{split} \end{equation*}
and choosing $\nu_0 < 1$ and  $\theta^2 = \frac{1-\nu_0}{2}$, we obtain such that
\begin{equation*}
 \int_{\mathbb{B}_r(p)} u^2   \left\vert \nabla \mu^{\frac{n-\nu_0}{2}} \right\vert^2_g 
\leq  \frac{4}{(1-\nu_0)^2} \int_{\mathbb{B}_r(p)} \mu^{n-\nu_0}  \left\vert \nabla u \right\vert^2_g + \frac{2}{(1-\nu_0)(n-1)} \int_{\mathbb{B}_r(p)} u^2 \mu^{n-\nu_0}(B+k_g) 
 \end{equation*}
 and passing to the limit as $\tau \to 0$, we have that $\mu^{2} \to f$ as $\tau\to 0$ and noticing that $n-\nu_0 = \frac{n}{\kappa_0}$ we obtain
 \begin{equation*}
 \int_{\mathbb{B}_r(p)} u^2  \left\vert \nabla f^{\frac{n}{4\kappa_0}} \right\vert^2_g 
\leq   \frac{4}{(1-\nu_0)^2}  \int_{\mathbb{B}_r(p)} f^{\frac{n}{2\kappa_0}} \left\vert \nabla u \right\vert^2_g + \frac{2}{(1-\nu_0)(n-1)} \int_{\mathbb{B}_r(p)} u^2 f^{\frac{n}{2\kappa_0}}(f+k_g) 
 \end{equation*}
so that there is a constant $c:= c(n, 1-\nu_0) $ such that
\begin{equation*}
\begin{split}  \left(\int_{\mathbb{B}_r(p)} u^{2\kappa_0} f^{\frac{n}{2}} dv_g \right)^{\frac{1}{\kappa_0}} \leq  &  2 C_0 r^{2-\frac{n(\kappa_0-1)}{\kappa_0}}   \int_{\mathbb{B}_r(p)} \left\vert \nabla  u \right\vert_g^2 f^{\frac{n}{2\kappa_0}} ( 1 + c)  dv_g  \\
 &  +  2 C_0 r^{2-\frac{n(\kappa_0-1)}{\kappa_0}}   c  \int_{\mathbb{B}_r(p)} u^2 f^{\frac{n}{2\kappa_0}}(f+k_g) dv_g \\
 \leq &  2 C_0 r^{2-\frac{n(\kappa_0-1)}{\kappa_0}}  A^{1-\frac{n(\kappa_0-1)}{2\kappa_0}} \left(1+c \right) \int_{\mathbb{B}_r(p)} \left\vert \nabla  u \right\vert_g^2 f^{\frac{n-2}{2}}dv_g \\
 & + 2 C_0 c_g r^2 c \left(A+k_g\right)  \left(\int_{\mathbb{B}_r(p)} u^{2\kappa_0} f^{\frac{n}{2}} dv_g \right)^{\frac{1}{\kappa_0}}
\end{split}
\end{equation*}
noticing that $f$ is bounded by a constant $A$ in $M$ and applying a H\"older inequality for the last inequality. Letting $r>0$ small enough so that $2 C_0 c_g r^2 c \left(A+k_g\right)  < \frac{1}{2}$, we obtain the expected Sobolev inequality.
\end{proof}

\begin{rem} In the classical Sobolev inequality, the optimal constant $2\kappa_0$ is $\frac{2n}{n-2}$. 
Our Caccioppoli type estimate involved in the proof of the Sobolev inequality seems to be optimal regarding other similar regularity results on the $n$-harmonic equation (e.g \cite{Sarsa}). In this case, we only obtain $2\kappa_0 < \frac{2n}{n-1}$.
It would be interesting to improve $\nu_0 >0$ so that 
$f^{\frac{n-\nu_0}{4}}$ is a $H^1$ function. 
\end{rem}

Now, it is clear that the embedding $W^{1,2}(f . g) \to L^2(f . g)$ is compact. Indeed, first, any sequence of functions $(u_k)$ bounded in $W^{1,2}(f . g)$ has to be bounded in $L^{2\kappa_0}(f . g)$ (up to take a partition of unity to globalize the previous Sobolev inequality). We also know that up to a subsequence, it converges weakly to $u\in W^{1,2}(f . g)$. Moreover, using the classical compact embedding $W^{1,2}(g) \to L^2(g)$ on any compact subset of $M \setminus Z$ where $Z = \{ x\in M; f(x)=0\}$, we deduce that up to a subsequence, the sequence $(u_k)$ converges almost everywhere to $u$ with respect to the measure $f^{\frac{n}{2}} dv_g$. These two properties prove that up to a subsequence $(u_k)$ converges strongly in $L^2(f . g)$. 

\subsection{Conclusion}
Let's prove Theorem \ref{theosplit}. By section \ref{sec1}, there is a Palais-Smale sequence for the maximization problem. We denote $\lambda_\eps^{\frac{n}{2}} = \lambda_k(Q_\eps,e^{nu_\eps})$. We apply proposition \ref{palaissmale} to this sequence. Theorem \ref{theosplit} then follows from upper semi-continuity of $\bar\lambda_k$. Indeed, let $\theta_0,\cdots,\theta_k$ be the $(k+1)$ first eigenfunctions of the limiting manifold endowed with generalized metrics $(\widetilde{M},f.g) = (M,f_0.g) \sqcup  \left(\mathbb{S}^n, f_1.h \right)\sqcup \cdots \sqcup \left(\mathbb{S}^n,f_t.h\right) $ if $f_0\neq 0$ or $ (\widetilde{M},f.g) =  \left(\mathbb{S}^n,f_1.h\right)\sqcup \cdots \sqcup \left(\mathbb{S}^n,f_t.h\right) $ if $f_0 = 0$, where $f_0 dv_g = \tilde{\mu}_0$ and $f_i dv_h = \pi^{\star}\tilde{\mu}_i$ (where $\pi : \mathbb{S}^n \to \mathbb{R}^n \cup \{\infty\}$ is a stereographic projection) and $f_i$ are $\mathcal{C}^{0,\alpha}$ functions. We use $\theta_0,\cdots,\theta_k$ as test functions in the variational characterization of $\lambda_\eps$ in the following way: we set
$$ \tilde{\theta}_i^\eps(x) = \eta_0(x) \theta_i(x) + \eta_1\left( \frac{x- q_1^\eps}{\alpha_1^\eps} \right) \theta_i \circ \pi^{-1}\left( \frac{x- q_1^\eps}{\alpha_1^\eps} \right) + \cdots + \eta_t \left( \frac{x- q_t^\eps}{\alpha_t^\eps} \right) \theta_i\circ \pi^{-1}\left( \frac{x- q_t^\eps}{\alpha_t^\eps} \right) $$
where points $q_i^\eps$ and scales $\alpha_i^\eps$ are given by Claim \ref{clbubbletree} and we define $\eta_i \in \mathcal{C}_c^\infty\left(\Omega_i(\rho) \right)$ as cut-off functions such that $\eta_i = 1$ in $ \mathcal{C}^\infty\left(\Omega_i(\sqrt{\rho}) \right)$ and 
$$\int_{M} \left\vert \nabla \eta_0 \right\vert_g^n \leq \frac{C}{\ln{\frac{1}{\rho}}} \text{ and } \int_{\mathbb{R}^n} \left\vert \nabla \eta_i \right\vert_g^n \leq \frac{C}{\ln{\frac{1}{\rho}}} \text{ if } i \geq 1 $$
where if we denote $p_1^i, \cdots, p_1^{s_i}$ the rescaled bad points,
$$ \Omega_0(\rho) = \Sigma \setminus \bigcup_{j=1}^{s_0} \mathbb{B}_\rho\left(p_i^j \right) = \text{ and } \Omega_i(\rho) = \mathbb{B}_{\frac{1}{\rho}}\setminus \bigcup_{j=1}^{s_i} \mathbb{B}_\rho\left(p_i^j \right) \text{ if } i \geq 1. $$
Then, knowing that $Q_\eps\left( q_i^\eps + \alpha_i^\eps x \right) \to f^{\frac{n-2}{2}}$ in $L^{\frac{n}{n-2}}\left(\Omega_i(\rho)\right)$ and $\lambda_\eps^{\frac{n}{2}}e^{nu_\eps}\left( q_i^\eps + \alpha_i^\eps x \right) \to f^{\frac{n}{2}}$ for the weak-$\star$ convergence of measures on $\Omega_i(\rho)$, a straightforward computation gives that 
$$ \lambda_\eps \leq \max_{\varphi \in \left\langle \tilde{\theta}_0^\eps,\cdots,\tilde{\theta}_k^\eps \right\rangle} \frac{\int_M \left\vert \nabla \varphi \right\vert_g^2 \lambda_\eps^{\frac{2-n}{2}}Q_\eps dv_g}{ \int_M  \varphi^2 e^{n u_\eps} dv_g} \leq  \lambda_k\left( \tilde{M}, f.h \right) + o(1) + c \left(\frac{C}{\ln{\frac{1}{\rho}}}\right)^{\frac{1}{n}} $$
as $\eps \to 0$. Letting $\eps \to 0$ and then $\rho \to 0$ gives
$$ \limsup_{\eps\to 0} \lambda_\eps \leq  \lambda_k\left( \tilde{M}, f.h \right) $$
then $\Lambda_k(M, [g]) \leq \lambda_k\left( \tilde{M}, f.h \right) $. By a result by Colbois and El-Soufi \cite{ces} (see also \cite{fs20} in Steklov case), we must have equality, so that the limit of $\lambda_\eps$ is the $k$-th eigenvalue of the limit of the maximizing sequence. We also obtain   Theorem \ref{theomain}: with the strict inequality assumption, there is no bubbling ($t=0$).

\section{Independance of regularity estimates for harmonic maps with respect to the dimension of the target sphere} \label{sec3}

In this section, we aim at generalizing the known estimates on $n$-harmonic maps, noticing carefully their possible dependance on the dimension of the target sphere. 

\subsection{A priori estimates for $n$-harmonic maps from a $n$-manifold into spheres}

We say that $\Psi : \mathbb{B}^n \to \mathbb{S}^p$ is a $\tau$-approximated $n$-harmonic map if it is the limit as $\tau \to 0$ of a sequence of maps $\Psi_{\tau} : \mathbb{B}^n \to \mathbb{S}^p$ which are solutions of the following minimization problems
$$ \inf \left\{ \int_{\mathbb{B}^n} \left(\left\vert \nabla \varphi \right\vert_g^2 + \tau\right)^{\frac{n}{2}} dv_g ; \varphi :  \mathbb{B}^n \to \mathbb{S}^p \text{ and } \varphi = \Psi \text{ on } \partial\mathbb{B}^n  \right\} $$
It satisfies the following Euler-Lagrange equation
$$ -div_g\left( \left(\left\vert \nabla \Psi_\tau \right\vert_g^2 + \tau \right)^{\frac{n-2}{2}} \nabla \Psi_\tau \right) = \left(\left\vert \nabla \Psi_\tau \right\vert_g^2 + \tau \right)^{\frac{n-2}{2}} \left\vert \nabla \Psi_\tau \right\vert_g^2 \Psi_\tau  $$
and we have that $\Psi_{\tau} \in \mathcal{C}^{\infty}\left( \mathbb{B}^n \right)$\cite{HardtLin}\cite{Strzelecki}. We say that $\Psi_\tau$ is a $(\tau,n)$-harmonic map. For instance, we will deduce from Proposition \ref{regular n-harm} and \cite{Strzelecki} that any $n$-harmonic map into a sphere is locally a $\tau$-approximated $n$-harmonic map. In particular, we deduce a local uniqueness of the harmonic replacement of $n$-harmonic maps into a sphere $\mathbb{S}^p$. We also have the following proposition:

\begin{prop} \label{propc0estnharm}
There is a constant $\eps_0$ and a constant $C_{n}$ such that for any $p\geq 2$ and any $\tau$-approximated $n$-harmonic map $\Psi: \mathbb{B}_n \to \mathbb{S}^p$ such that 
$$ \int_{ \mathbb{B}^n} \left\vert \nabla \Psi \right\vert_g^n \leq \eps_0 $$
we have for any ball $\mathbb{B}^n_r(p)\subset \mathbb{B}_n$ 
$$ r^n \left\| \nabla \Psi \right\|_{\mathcal{C}^{0}\left(\mathbb{B}^n_{\frac{r}{2}}(p)\right)}^n \leq C_n \int_{\mathbb{B}^n_r(p)} \left\vert \nabla \Psi \right\vert_g^n dv_g $$

\end{prop}
In order to prove Proposition \ref{propc0estnharm}, we first prove that a $(\tau,n)$-harmonic map satisfies the following 
\begin{cl} \label{clapproxdeltanharm}
For any $p\geq 2$ and any $(\tau,n)$-harmonic map $\Psi_ \tau : \mathbb{B}^n \to \mathbb{S}^p$, there is $(a_{i,j}^\tau)_{1\leq i,j \leq n}$ such that for any $X \in \R^n$
$$ \frac{1}{n} \left\vert X \right\vert^2 \leq a_{i,j}^\tau X^iX^j \leq  \frac{n-1}{n} \left\vert X \right\vert^2  \text{ and } \forall i,j, \left\vert a_{i,j}^\tau \right\vert \leq 1$$ 
and such that
$$ u\left(\kappa_g  + \left\vert \nabla \Psi \right\vert_g^2\right) \geq P \left\vert \nabla \Psi \right\vert^2  \left(\kappa_g  + \left\vert \nabla \Psi \right\vert_g^2\right) \geq  -   div_g\left( a^\tau \nabla u \right) +  \left\vert \nabla^2 \Psi \right\vert P +\frac{n-2}{4} \mu^{n-4} \left\vert \nabla B \right\vert^2 $$
where $B = \left\vert \nabla \Psi_\tau \right\vert_g^2 $, $\mu =  \left(B + \tau \right)^{\frac{1}{2}} $, $P= \mu^{n-2}$, 
$ u = \mu^n $ and $\kappa_g$ is a constant such that $Ric_g \geq -\kappa_g g$.
\end{cl}

\begin{proof}
During the proof we set $\Psi = \Psi_\tau$, $u=u_\tau$ and $B= \left\vert \nabla \Psi \right\vert_g^2$ and $P = \left(B+\tau\right)^{\frac{n-2}{2}}$ so that $u = PB$. We have that
\begin{equation*} \begin{split} & \nabla \Psi . \Delta\left( P \nabla \Psi \right)  \\ = & \sum_i \left( - \nabla_i\left( \nabla \Psi . \nabla_i\left(P\nabla \Psi \right)  \right) + \nabla_i\nabla\Psi . \nabla_i\left( P \nabla \Psi \right) \right) \\
= & - \sum_i \left( \nabla_i\left( \left\vert \nabla \Psi \right\vert^2 \nabla_i P  \right) + \nabla_i\left( \nabla_i \nabla \Psi . \nabla \Psi P \right) \right) \\
& + \sum_i \left( P \nabla_i \nabla \Psi . \nabla_i \nabla \Psi + \nabla_i P  \nabla_i \nabla \Psi . \nabla\Psi \right) \\
=& - \sum_i  \nabla_i\left( \left( \frac{n-2}{2}B \left(B+\tau\right)^{\frac{n-4}{2}} + \frac{1}{2} P \right) \nabla_i B \right) \\
& + \left\vert \nabla^2 \Psi \right\vert P +\frac{n-2}{4} \left(B+\tau\right)^{\frac{n-4}{2}} \left\vert \nabla B \right\vert^2 \\
\end{split}
\end{equation*}
Now, we aim at computing $\nabla \Psi . \Delta\left( P \nabla \Psi \right)$ in another way, using the equation on the $n$-harmonic map $\Psi$. We write $\Delta_g = - d_g ^\star d_g - d_g d_g^\star$ the Laplacian acting on forms, so that
\begin{equation*}
\begin{split} 
& \left\langle d\Psi , \Delta_g\left( P d \Psi \right) \right\rangle_g =  - \left\langle d\Psi , \left(d_g^\star d_g + d_gd_g^\star\right)\left(  P d \Psi \right) \right\rangle_g  \\
= & -\left\langle d\Psi , d_g^\star\left( dP  \wedge d\Psi \right)\right\rangle_g + \left\langle d\Psi , d\left(   P \left\vert \nabla \Psi \right\vert_g^2 \Psi \right) \right\rangle_g + P \left\langle d\Psi , \star\left( (\star F) \wedge  d \Psi) \right) \right\rangle_g \\
= & - \left\langle d\Psi , d_g^\star\left( dP  \wedge d\Psi \right)\right\rangle_g + P \left\vert \nabla \Psi \right\vert_g^4 - P Ric_g(\nabla \Psi, \nabla \Psi)
\end{split}
 \end{equation*}
 where we used for the second equality that $d_g d_g\Psi = F \Psi$, where $F$ is the curvature 2-form, that $- d_g^\star\left( P d\Psi\right)  = P \left\vert \nabla \Psi \right\vert^2 \Psi $ and that $\Psi . d\Psi = 0 $. We use again $\Psi . d\Psi = 0 $ for the third equality. Now, for a function $\theta \in \mathcal{C}^{\infty}_c\left(\mathbb{B}^n\right)$, we have
\begin{equation*}
\begin{split} - \int_{\mathbb{B}^n} \theta & \left\langle d_g\Psi , d_g^\star\left( d P  \wedge d\Psi \right)\right\rangle_g dv_g =  \int_{\mathbb{B}^n} \left\langle d\theta \wedge d\Psi ,  dP  \wedge d\Psi \right\rangle_g dv_g \\
 & =  \int_{\mathbb{B}^n} \left( \left\langle d\theta  ,  d P  \right\rangle_g B - \left\langle d\theta , d\Psi \right\rangle_g \left\langle  dP  , d\Psi \right\rangle_g  \right)dv_g \\
 & = -  \int_{\mathbb{B}^n} \theta \nabla_i \left( \frac{n-2}{2}  \left(B+\tau\right)^{\frac{n-4}{2}}\left( B \delta_{i,j}   -  \nabla_i \Psi \nabla_j \Psi \right) \nabla_j B\right)
\end{split} 
\end{equation*}
where we used again for the second equality that $d_g d_g\Psi = F \Psi$ and $\Psi.d\Psi= 0$. Combining the previous equalities
\begin{equation*} \begin{split} P \kappa_g \left\vert \nabla \Psi \right\vert_g^2 +  P  \left\vert \nabla \Psi \right\vert_g^4 \geq & - \sum_i  \nabla_i\left( \left( \frac{n-2}{2}B \left(B+\tau\right)^{\frac{n-4}{2}} + \frac{1}{2} P \right) \nabla_i B \right) \\
& + \sum_{i,j} \frac{n-2}{2}  \nabla_i \left(  \left(B+\tau\right)^{\frac{n-4}{2}} \left(B\delta_{i,j} - \nabla_i \Psi \nabla_j \Psi \right) \nabla_j B \right) \\
& + \left\vert \nabla^2 \Psi \right\vert P +\frac{n-2}{4} \left(B+\tau\right)^{\frac{n-4}{2}} \left\vert \nabla B \right\vert^2
\end{split} \end{equation*}
Noticing that 
$$ \nabla u = \frac{n}{2} P \nabla B  $$
we obtain that
\begin{equation*} u  \left(\kappa_g  + \left\vert \nabla \Psi \right\vert_g^2\right) \geq  -   div_g\left( a^\tau \nabla u \right) +  \left\vert \nabla^2 \Psi \right\vert P +\frac{n-2}{4} \left(B+\tau\right)^{\frac{n-4}{2}} \left\vert \nabla B \right\vert^2
 \end{equation*}
 where 
\begin{equation*} a_{i,j}^\tau =  \frac{1}{n}\left( \delta_{i,j} + (n-2) \frac{\nabla_i \Psi.\nabla_j \Psi}{B + \tau} \right)
 \end{equation*}
\end{proof}

We deduce $\eps$-regularity results independant of the dimension of the target sphere on these maps:

\begin{cl} \label{clindependanttargetmanifoldepsreg} For any $n\geq 3$, there is $\eps_{n,g} >0$ and a constant $C_{n,g}$ such that for any $p\geq 2$, any $\tau >0$ and any $(\tau,n,g)$-harmonic map $\Psi_ \tau : \mathbb{B}^n \to \mathbb{S}^p$ such that 
$$ \int_{ \mathbb{B}^n} \left(\left\vert \nabla \Psi_{\tau} \right\vert_g^2 + \tau\right)^{\frac{n}{2}} dv_g \leq \eps_{n,g} $$
we have for any ball $\mathbb{B}^n_r(p)\subset \mathbb{B}_n$ 
$$ r^n \left\| \nabla \Psi_\tau \right\|_{\mathcal{C}^{0}\left(\mathbb{B}^n_{\frac{r}{2}}(p)\right)}^n \leq {C_{n,g}} \int_{\mathbb{B}^n_r(p)} \left(\left\vert \nabla \Psi_{\tau} \right\vert_g^2 + \tau\right)^{\frac{n}{2}} dv_g $$

\end{cl}

\begin{proof} Notice that it is sufficient to prove that there is $\eps_n$ small enough such that for any $p \in \mathbb{B}^n$ and $r_0>0$ such that $\mathbb{B}^n_{r_0}(p) \subset \mathbb{B}_n$
$$  \left\vert  \nabla \Psi_\tau(p) \right\vert^n \leq \frac{\delta}{ r_0^n} \text{ where }  \int_{\mathbb{B}^n_{r_0}(p)} \left(\left\vert \nabla \Psi_{\tau} \right\vert_g^2 + \tau\right)^{\frac{n}{2}} dv_g = \delta  \int_{ \mathbb{B}^n} \left(\left\vert \nabla \Psi_{\tau} \right\vert_g^2 + \tau\right)^{\frac{n}{2}} dv_g \leq \delta \eps_{n,g} $$
We set 
$$F(x) = \left( r_0 - \left\vert x- p \right\vert \right)^n \left\vert \nabla \Psi_\tau   \right\vert_g^n$$ 
and we let $x_0 \in \mathbb{B}^n_{r_0}(p)$ be such that $F(x_0) = \sup_{x\in \mathbb{B}^n_{r_0}(p)} F(x) $. Notice that it is sufficient to prove that for $\eps$ small enough, $F(x_0)\leq \delta$. We assume by contradiction that $F(x_0) > \delta $. We set $\sigma>0$ such that
$$ \sigma^n  \left\vert \nabla \Psi_{\tau} \right\vert_g^n(x_0) = \frac{\delta}{4} $$
Since $F(x_0) >\delta$, we have that $2\sigma \leq r_0 - \left\vert x-p \right\vert$. By a triangle inequality, we have for $x \in \mathbb{B}_\sigma(x_0)$ that
$$ \frac{1}{2} \leq \frac{r_0 - \left\vert x-p \right\vert}{r_0 - \left\vert x_0-p \right\vert} \leq 2 $$
Since $F$ realizes its maximum at $x_0$, we have
\begin{equation*}
\begin{split}
\left(r_0 - \left\vert x-x_0 \right\vert\right)^n \sup_{x\in\mathbb{B}^n_\sigma(x_0)} \left\vert  \nabla \Psi_\tau(x) \right\vert^n  \leq 4 \sup_{x\in \mathbb{B}^n_\sigma(x_0)} F(x) =  4 F(x_0) \\ 
\leq 4 \left(r_0 - \left\vert x-x_0 \right\vert\right)^n \left\vert  \nabla \Psi_\tau(x_0) \right\vert^n 
\end{split}
\end{equation*}
so that by definition of $\sigma$
\begin{equation} \label{eqestgradpsitausigma} \sup_{x\in\mathbb{B}^n_\sigma(x_0)} \left\vert  \nabla \Psi_\tau(x) \right\vert^n \leq 4 \left\vert  \nabla \Psi_\tau(x_0) \right\vert^n = \frac{\delta}{\sigma^n}. \end{equation}
We set $\tilde{g}(z) = g(x_0+\sigma z)$, $\tilde{\tau} = \sigma^2 \tau$, $\widetilde{\Psi}(z) = \Psi( x_0 + \sigma z )$ and $\widetilde{Q} = \left( \left\vert \nabla \widetilde{\Psi} \right\vert^2 + \tilde{\tau} \right)^{\frac{n-2}{2}}$ and $\tilde{u} = \left( \left\vert \nabla \widetilde{\Psi} \right\vert^2 + \tilde{\tau} \right)^{\frac{n}{2}}$ so that we obtain 
$$ -div_{\tilde{g}}\left( \widetilde{Q} \nabla \widetilde{\Psi} \right) = \widetilde{Q}  \left\vert \nabla \widetilde{\Psi} \right\vert^2  \widetilde{\Psi} $$
and by Claim \ref{clapproxdeltanharm},
$$ -div_{\tilde{g}}(\tilde{a} \nabla \tilde{u}) \leq \tilde{u} \left( \left\vert \nabla \widetilde{\Psi} \right\vert_{\tilde{g}}^2 + \kappa_{\tilde{g}} \right) $$
where $\left(\tilde{a}_{i,j}\right)$ is such that for any $X \in \R^n$
$$ \frac{1}{n} \left\vert X \right\vert^2 \leq \tilde{a}_{i,j} X^iX^j \leq  \frac{n-1}{n} \left\vert X \right\vert^2  \text{ and } \forall i,j, \left\vert \tilde{a}_{i,j} \right\vert \leq 1$$ 
By the rescaled version of  \eqref{eqestgradpsitausigma} , we obtain
$$ -div_{\tilde{g}}(\tilde{a} \nabla \tilde{u}) \leq \tilde{u} \left(1 + \kappa_{\tilde{g}} \right) $$
By standard elliptic a priori estimates for smooth positive subsolutions (see e.g \cite{HL} chapter 4) and knowing that $\tilde{u} \geq 0$, we obtain that
$$ \frac{\delta}{4} = \left\vert \nabla \widetilde{\Psi}\right\vert_{\tilde{g}}^n(0) \leq \tilde{u}(0) \leq K_{n} \left(1 + \kappa_{\tilde{g}} \right) \int_{\mathbb{B}^n} \tilde{u}  \leq K_{n,g} \int_{\mathbb{B}^n_{r_0}(p)} u_\tau dv_g \leq K_{n,g} \delta \eps_{n,g} $$
for some constants $K_n$ and $K_{n,g}$. Setting $\eps_{n,g} = \frac{1}{8K_{n,g}}$ gives the Claim.
\end{proof}

Proposition \ref{propc0estnharm} then follows letting $\tau \to 0$ on the $(\tau,n)$ harmonic maps that converge to the $\tau$-approximated $n$-harmonic map.

\subsection{Global strong convergences independant from the dimension of the target manifold}

The following claim adapts a result by Courilleau \cite{Cou} to infinite dimensional target manifolds.

\begin{cl} \label{clW1qconvergence} Let $u_k : M \to \mathbb{R}^{\mathbb{N}}$ be a sequence of maps such that
$$\limsup_{k\to +\infty} \int_{M} \left( \left\vert u_k \right\vert^2  + \left\vert \nabla u_k \right\vert_g^n \right)dv_g < +\infty $$
We assume that
$$ div_g\left( \left\vert \nabla u_k \right\vert_g^{n-2} \nabla u_k^i \right) = A_k^i + B_k^i $$
$$ \limsup_{k\to +\infty}  \left( \int_M \left\vert A_k \right\vert + \| B_k \|_{W^{-1,n}} \right)< +\infty $$
and
$$  \| B_k \|_{W^{-1,n}} \to 0 \text{ as } k \to +\infty  $$
where we set
$$ \left\vert u_k \right\vert^2 = \sum_{i=0}^{+\infty} \left(u_k^i\right)^2 \text{ and } \left\vert \nabla u_k \right\vert_g^{2} = \sum_{i=0}^{+\infty} \left\vert \nabla u_k^i \right\vert_g^2 $$
and for $n_\star = \frac{n}{n-1}$,
$$ \| B_k \|_{W^{-1,n}} = \sup_{ \varphi : M \to \mathbb{R}^{\mathbb{N}}  } \frac{\left\vert \int_M \sum_i B_k^i \varphi_i dv_g \right\vert }{ \left( \int_M \left(\left\vert \varphi \right\vert^{n_\star} + \left\vert \nabla \varphi \right\vert^{n_\star}\right)dv_g\right)^{\frac{1}{n_\star}} } $$
Then, up to a subsequence, there is $u : M \to \mathbb{R}^{\mathbb{N}}$ such that for any $1\leq p < +\infty$ and any $1\leq q < n$, 
\begin{equation} \label{eqLqconvergencegrad} \int_M \left( \left\vert u_k -u  \right\vert^p  + \left\vert \nabla \left(u_k-u\right) \right\vert_g^{q} \right)dv_g \to 0 \end{equation}
as $k\to +\infty$.
\end{cl}

\begin{proof}
We first notice that by classical compact Sobolev embeddings, up to a diagonal extraction of subsequences, there is a subsequence of $(u_k)$ such that for any $i \in\mathbb{N}$, we have
$$ u_k^i \rightharpoonup u^i \text{ in } W^{1,n} $$
$$ u_k^i \to u_i \text{ in } L^p \text{ for any } 1\leq p < +\infty $$
$$ u_k^i \to u_i \text{ a.e }$$
$$ \forall k \in \mathbb{N}, \int_M \left(u_k^i - u^i \right)^2 dv_g \leq 2^{-i}$$
$$ \forall (\varphi_i), \int_M \left(\left\vert \varphi \right\vert^{n_\star} + \left\vert \nabla \varphi \right\vert^{n_\star}\right)dv_g < +\infty \Rightarrow \int_M \sum_i \nabla \left(u_k^i - u^i \right) \nabla \varphi_i \to 0 \text{ as } k\to +\infty$$
Then, by weak convergence, we have that for any $m\in \mathbb{N}$
$$ \int_M \left(\sum_{i=0}^m \left\vert \nabla u^i \right\vert_g^2\right)^{\frac{n}{2}}dv_g  \leq \liminf_{k\to+\infty} \int_M \left(\sum_{i=0}^m \left\vert \nabla u_k^i \right\vert_g^2\right)^{\frac{n}{2}}dv_g \leq   \liminf_{k\to+\infty} \int_M \left\vert \nabla u_k \right\vert_g^ndv_g $$
so that passing to the limit as $m\to +\infty$, we have
$$  \int_M \left\vert \nabla u \right\vert_g^ndv_g \leq   \liminf_{k\to+\infty} \int_M \left\vert \nabla u_k \right\vert_g^ndv_g. $$
By the same argument, using Sobolev embeddings, we have for any $1 \leq p< +\infty$
$$ \int_M \left\vert u \right\vert^p dv_g \leq \liminf_{k\to +\infty} \int_M \left\vert u_k \right\vert^p dv_g $$
In fact, we have equality in the previous inequality. Indeed, for any $m\in \mathbb{N}$
$$ \lim_{k\to +\infty } \int_M \left(\sum_{i=0}^m \left( u_k^i - u^i \right)^2\right)^{\frac{p}{2}} dv_g = 0 $$
so that since 
$$ \left\vert u_k - u \right\vert^2 \leq  \sum_{i=0}^m \left( u_k^i - u^i \right)^2 + 2^{-m}  $$
we obtain by a Holder inequality that for any $m\in \mathbb{N}$,
$$ \lim_{k\to +\infty } \int_M  \left\vert u_k - u \right\vert^pdv_g \leq 2^{-m} $$
and we obtain the first part of \eqref{eqLqconvergencegrad}. Now let's prove \eqref{eqLqconvergencegrad}.

\textbf{STEP 1: } Up to a new subsequence, we have that
\begin{equation} \label{eqstep1courilleau} \sum_{i=0}^{+\infty}\left( \left\vert \nabla u_k \right\vert_g^{n-2} u_k^i - \left\vert \nabla u \right\vert_g^{n-2} \nabla u^i \right).\left( \nabla u_k^i - \nabla u^i \right) \to_{a.e} 0 \text{ as } k\to +\infty \end{equation}

\textbf{Proof of Step 1:} we have that 
\begin{equation*}
\begin{split} \sum_{i=0}^{+\infty}\left( \left\vert \nabla u_k \right\vert_g^{n-2} u_k^i - \left\vert \nabla u \right\vert_g^{n-2} \nabla u^i \right).\left( \nabla u_k^i - \nabla u^i \right) \\
 \geq \left( \left\vert \nabla u_k \right\vert_g -\left\vert \nabla u \right\vert_g \right)\left( \left\vert \nabla u_k \right\vert_g^{n-1} -\left\vert \nabla u \right\vert_g^{n-1} \right) \geq 0
\end{split}
 \end{equation*}
Then, we just need upper bounds. 
Let $\delta>0$. By Egoroff 's theorem there is $E_{\delta}\subset\subset M$ such that $Vol_g(M\setminus E_\delta) < \delta$ and such that $\left\vert u_k - u \right\vert^2$ converges uniformly to $0$ in $E_\delta$. We aim at proving that
\begin{equation} \label{convergesto0egoroff} \lim_{k\to +\infty} \int_{E_\delta}  \sum_{i=0}^{+\infty}\left( \left\vert \nabla u_k \right\vert_g^{n-2} u_k^i - \left\vert \nabla u \right\vert_g^{n-2} \nabla u^i \right).\left( \nabla u_k^i - \nabla u^i \right)dv_g = 0 \end{equation}
Let $\eps>0$. We let $\delta_\eps>0$ be such that for any set such that $Vol_g(M\setminus A) < \delta_\eps$
$$ \int_{M\setminus A} \left\vert \nabla u \right\vert^n dv_g < \eps $$
and we use this with $A_{\delta_\eps} = E_{\delta_\eps} \cup E_\delta$. By uniform convergence, let $k_0$ be such that for any $k\geq k_0$, $\left\vert u-u_k \right\vert \leq \eps$ on $A_{\delta_\eps}$. We then have that for 
a cut-off function $\eta \in \mathcal{C}^{\infty}_c(M)$ such that $\eta \leq 1$ and $\eta = 1$ in $A_{\delta_\eps}$
a truncation function $\beta_\eps : \mathbb{R}^{\mathbb{N}} \to \mathbb{R}^{\mathbb{N}} $ defined as $\beta_\eps(v) = v$ if $\left\vert v \right\vert \leq \eps$ and $\beta_\eps(v) = \eps \frac{v}{\left\vert v \right\vert}$ if $\left\vert v \right\vert >\eps$.
\begin{equation*}
\begin{split} \int_{M} & \eta  \sum_{i=0}^{+\infty}  \left( \left\vert \nabla u_k \right\vert_g^{n-2} \nabla u_k^i - \left\vert \nabla u \right\vert_g^{n-2} \nabla u^i \right).\left( \nabla \beta_\eps(u_k - u)^i \right)dv_g  \\
\leq & \int_M \eta \sum_i \left(A_k^i + B_k^i\right) \beta_{\eps}\left(u_k - u\right)^i - \int_M \sum_i \left( \nabla \eta \left\vert \nabla u_k \right\vert_g^{n-2} \nabla u_k^i \right) \beta_{\eps}\left(u_k - u\right)^i \\
& - \int_{M} \eta \sum_{i} \left\vert \nabla u \right\vert_g^{n-2} \nabla u^i . \nabla \left( \beta_{\eps}\left(u_k - u\right)^i \right)  \\
\leq & \left( \int_M \left\vert A_k \right\vert  \right)  \left\| \left\vert \beta_{\eps}\left(u_k - u \right) \right\vert \right\|_{\infty} +  \left\|  \nabla \eta \right\|_{\infty} \left(\int_M  \left\vert \nabla u_k \right\vert^{n}\right)^{\frac{n-1}{n}} \left(\int_{M} \left\vert \beta_\eps(u_k-u) \right\vert^n\right)^{\frac{1}{n}} \\
&+  \| B_k \|_{W^{-1,n}} \| \left\vert \beta_\eps(u_k-u) \right\vert \|_{W^{1,n}} + \left\vert \int_{M} \eta \sum_{i} \left\vert \nabla u \right\vert_g^{n-2} \nabla u^i . \nabla \left( \beta_{\eps}\left(u_k - u\right)^i \right) \right\vert  \\
\leq &  \left( \int_M \left\vert A_k \right\vert  \right)  \eps + o(1) \text{ as } k \to +\infty
\end{split}
 \end{equation*}
We also have that
\begin{equation*}
\begin{split} & \left\vert \int_{M\setminus A_{\delta_\eps}}  \eta  \sum_{i=0}^{+\infty}  \left( \left\vert \nabla u_k \right\vert_g^{n-2} \nabla u_k^i - \left\vert \nabla u \right\vert_g^{n-2} \nabla u^i \right)\left( \nabla \beta_\eps(u_k - u)^i \right)dv_g \right\vert \\
& \leq \int_{M\setminus A_{\delta_\eps}}\left( \left\vert \nabla u_k \right\vert^{n-1} \left\vert \nabla \beta_\eps(u_k-u) \right\vert + \left\vert \nabla u \right\vert^{n-1} \left\vert \nabla \beta_\eps(u_k-u) \right\vert \right) \\
& \leq C \int_{M\setminus A_{\delta_\eps}}\left( \left\vert \nabla u_k \right\vert^{n-1} \left\vert \nabla u \right\vert + \left\vert \nabla u \right\vert^{n-1} \left\vert \nabla u_k \right\vert \right) \\
& \leq C \left( \left(\int_M \left\vert \nabla u_k \right\vert^{n}\right)^{\frac{n-1}{n}}  \eps^{\frac{1}{n}} +   \left(\int_M \left\vert \nabla u_k \right\vert^{n}\right)^{\frac{1}{n}}      \eps^{\frac{1}{n_\star}} \right)
\end{split}
 \end{equation*}
 and we finally obtain
\begin{equation*}
\begin{split} \limsup_{k\to +\infty} \int_{A_\delta} &  \sum_{i=0}^{+\infty}  \left( \left\vert \nabla u_k \right\vert_g^{n-2} \nabla u_k^i - \left\vert \nabla u \right\vert_g^{n-2} \nabla u^i \right).\left( \nabla (u_k^i - u^i) \right)dv_g  \\
\leq & \limsup_{k\to +\infty} \int_{A_{\delta_\eps}} \eta  \sum_{i=0}^{+\infty}  \left( \left\vert \nabla u_k \right\vert_g^{n-2} \nabla u_k^i - \left\vert \nabla u \right\vert_g^{n-2} \nabla u^i \right)\left( \nabla \beta_\eps(u_k - u)^i \right)dv_g \\
\leq & \limsup_{k\to +\infty}\left( \left( \int_M \left\vert A_k \right\vert  \right)  \eps  + C \left( \left(\int_M \left\vert \nabla u_k \right\vert^{n}\right)^{\frac{n-1}{n}}  \eps^{\frac{1}{n}} +   \left(\int_M \left\vert \nabla u_k \right\vert^{n}\right)^{\frac{1}{n}}      \eps^{\frac{1}{n_\star}} \right)  \right) 
\end{split}
 \end{equation*}
 Letting $\eps\to 0$ gives \eqref{convergesto0egoroff}. Then, letting $\delta\to 0$ gives Step 1.
 
 \medskip
 
 \textbf{STEP 2: } We have that
$$  \left\vert \nabla u_k - \nabla u \right\vert_g \to_{a.e} 0 \text{ as } k\to +\infty $$

\textbf{Proof of Step 2:} For a fixed $z \in M$ such that \eqref{eqstep1courilleau} occurs at the point $z$, we set $X_k = \nabla u_k(z) $, $X = \nabla u(z)$ and for any $Y \in \left(\mathbb{R}^2\right)^{\mathbb{N}}$,
$$ \left\vert Y \right\vert = \sum_{i=0}^{+\infty}g^{a,b}(z) \left(Y^i\right)_a \left(Y^i\right)_b   $$
We have that by Step 1 that
$$ \sum_{i=0}^{+\infty}\left( \left\vert X_k \right\vert^{n-2} X_k^i - \left\vert X \right\vert^{n-2} X^i \right).\left( X_k^i - X^i \right) \to 0  \text{ as } k \to +\infty$$
First, it is clear that $\left\vert X_k \right\vert$ is bounded. Indeed, we have deduce from the previous convergence that
$$ \left\vert X_k \right\vert^n + \left\vert X \right\vert^n \leq \left(\left\vert X_k \right\vert^{n-2} + \left\vert X \right\vert^{n-2}\right)\left\vert X_k \right\vert \left\vert X \right\vert + o(1) \text{ as } k \to +\infty. $$
Up to a subsequence there is $Z$ such that $X_k$ weak converges to $Z$ in $l^2\left(\left(\mathbb{R}^2\right)^{\mathbb{N}}\right)$. Up to another subsequence, we set $\alpha = \lim_{k \to +\infty} \left\vert X_k \right\vert$ and $\beta = \left\vert X \right\vert $. Then, passing to the limit, we have
$$ \alpha^n + \beta^n = \left(\alpha^{n-2} + \beta^{n-2}\right) \sum_i X^i Z^i $$
By Cauchy-Schwarz inequality and since $\left\vert Z \right\vert \leq \alpha$, we obtain
$$ \alpha^n + \beta^n \leq \left(\alpha^{n-2} + \beta^{n-2}\right) \alpha \beta $$
so that since $2\alpha \beta \leq \alpha^2 + \beta^2$,
$$ \alpha^n + \beta^n \leq \left( \alpha^2 \beta^{n-2} + \beta^2 \alpha^{n-2} \right)$$
which implies that
$$ \left(\alpha^{n-2} - \beta^{n-2}\right)\left(\alpha^2 -\beta^2\right) \leq 0 $$
and that all the previous inequalities are in fact equalities: $\alpha = \beta = \left\vert Z \right\vert $ and $\left\langle X,Z \right\rangle = \alpha^2$. In particular $Z = X$ and $X_k$ strongly converges to $Z = X$. Since the limit is independant from the subsequence, Step 2 is proved.

\medskip

 \textbf{STEP 3: } Conclusion: since $\left\vert \nabla u_k -\nabla u \right\vert$ is bounded in $W^{1,n}$ and converges almost everywhere to $0$ in $(M,g)$, then, $\left\vert \nabla u_k -\nabla u \right\vert$ strongly converges to $0$ in $L^q(M,g)$ for any $q < n$.

\end{proof}

\subsection{Point removability}

\begin{cl} \label{clpointremovability} Let $\Psi : \mathbb{B}\setminus\{0\} \to \mathbb{R}^{\mathbb{N}}$ be a $\mathcal{C}^1$ $n$-harmonic map into an infinite dimensional sphere such that
$$\int_{\mathbb{B}}  \left\vert \nabla \Psi \right\vert_g^n dv_g < +\infty \text{ and }  \left\vert \nabla \Psi \right\vert(x) \leq \frac{C}{\left\vert x \right\vert} $$ 
then $\Psi : \mathbb{B}\setminus\{0\} \to \mathbb{R}^{\mathbb{N}}$ extends to a $\mathcal{C}^{1,\alpha}$ function on $\mathbb{B}$ which is $n$-harmonic.
\end{cl}

We can follow and generalize the proof by \cite{mouyang} to infinite dimensional target spheres.

\subsection{Regular n-harmonic maps are locally $(\tau,n)$-approximated harmonic maps} \label{regular n-harm}

The aim of the section is to prove the following for $\mathcal{C}^1$ harmonic maps into a possibly infinite dimensional target sphere.

\begin{prop} \label{proptaunapproximatedharmmap} There is $\delta_{n,g}>0$ such that for any weak $n$-harmonic map $\Phi : M \to \mathbb{R}^{\mathbb{N}}$ into a possibly infinite dimensional target sphere, such that $\Phi \in \mathcal{C}^{1}$ and for any ball $\mathbb{B}_r(p)$ such that 
$$ \int_{\mathbb{B}_{2r}(p)} \left\vert \nabla \Phi \right\vert^n_g dv_g \leq \delta_{n,g} $$
for any  $\left(\Psi_{\tau,m}\right)_{\tau>0, n\in \mathbb{N}}$ minimizers of 
$$\Psi \mapsto \int_{\mathbb{B}_r(p)} \left( \left\vert \nabla \Psi\right\vert_g^2 + \tau \right)^{\frac{n}{2}} dv_g   $$
under the $W^{1,n}$ maps $\Psi : M \to \mathbb{S}^m$ such that $\left\vert \Psi \right\vert^2 =_{a.e} 1$ and $\Psi = \frac{(\phi_0,\cdots,\phi_m)}{\left(\sum_{i=0}^m (\phi_i)^2 \right)^{\frac{1}{2}}}$ on $\partial \mathbb{B}_{r}(p)$ we have that $\Psi_{\tau,m}$ converges to $\Phi$ as $(\tau,m) \to (0,+\infty)$.
\end{prop}

\begin{proof} 
\textbf{Step 1: } Since $\Phi$ is a $\mathcal{C}^1$ function, it satisfies the a priori estimate on the gradient for $x \in \mathbb{B}_r(p)$
$$\left\vert \nabla \Phi(x) \right\vert_g^2 \leq C \int_{\mathbb{B}_{2r}(p)} \left\vert \nabla \Phi \right\vert^n_g dv_g \leq C \delta_{n,g} $$
We have for $x \in \mathbb{B}_r(p)$ that
$$ \left\vert \phi_1(0) - \phi_1(x) \right\vert^2 \leq r^2 C_g \left\vert \nabla \Phi(z) \right\vert_g^2 \leq  C_g C \delta_{n,g}$$
so that taking  $\delta_{n,g}$ small enough, and up to a rotation of coordinates so that $\phi_1(0) = 1$ we can assume that 
$$ \forall x \in \mathbb{B}_r(p) \left\vert\phi_1(x) \right\vert \geq \frac{3}{4}. $$
Moreover, up to reduce $\delta_{n,g}$, and for $\tau$ small enough $\Psi_{\tau}$ has to satisfy 
$$ \forall x \in \mathbb{B}_r(p) \left\vert\left(\psi_{\tau,m}\right)_1(x) \right\vert \geq \frac{1}{2} $$
thanks to the same trick as in the proof of Step 1 of Claim \ref{clmain} based on Courant-Lebesgue lemma and Morrey-Sobolev embeddings.

\medskip

\textbf{Step 2: } We set $\tilde \Phi_m = \frac{\Phi_m}{\left\vert \Phi_m \right\vert}$ where $ \Phi_m = (\phi_0,\cdots,\phi_m, 0,\cdots)  $. We denote $\tilde{\Psi}_{\tau,m} : M \to \mathbb{R}^{\mathbb{N}}$ the minimizer of 
$$\Psi \mapsto \int_{\mathbb{B}_r(p)} \left( \left\vert \nabla \Psi\right\vert_g^2 + \tau \right)^{\frac{n}{2}} dv_g   $$
under the $W^{1,n}$ maps $\Psi : M \to \mathbb{S}^m$ such that $\left\vert \Psi \right\vert^2 =_{a.e} 1$ and $\Psi =  \tilde \Phi_m$ on $\partial \mathbb{B}_{r}(p)$. we also set 
$$\Psi_{\tau,m} = \left\vert \Phi_m \right\vert \tilde{\Psi}_{\tau,m} + \Phi- \Phi_m \text{ so that } \Psi_{\tau,m} - \Phi =  \left\vert \Phi_m \right\vert \left( \tilde{\Psi}_{\tau,m} - \tilde\Phi_m  \right).$$
We have:
\begin{equation*} 
\begin{split}\int_{\mathbb{B}_r(p)} \left\vert \nabla\left( \Phi -  \Psi_{\tau,m}\right) \right\vert_g^2 \left\vert \nabla \Phi \right\vert_g^{n-2} - \left(  \int_{\mathbb{B}_r(p)} \left\vert \nabla \Psi_{\tau,m} \right\vert_g^{2} \left\vert \nabla \Phi \right\vert_g^{n-2}  - \int_{\mathbb{B}_r(p)} \left\vert \nabla \Phi \right\vert_g^n \right) \\
= 2 \int_{\mathbb{B}_r(p)} \left\vert \nabla \Phi \right\vert_g^{n-2} \nabla \Phi . \nabla\left( \Phi -  \Psi_{\tau,m}\right) 
= 2 \int_{\mathbb{B}_r(p)} \left\vert \nabla \Phi \right\vert_g^{n}  \Phi . \left( \Phi -  \Psi_{\tau,m} \right) \\
=  \int_{\mathbb{B}_r(p)} \left\vert \nabla \Phi \right\vert_g^{n}  \left\vert \Phi -  \Psi_{\tau,m} \right\vert^2
\end{split}
\end{equation*}
and we obtain
\begin{equation*} 
\begin{split}\int_{\mathbb{B}_r(p)} \left\vert \nabla\left( \Phi -  \Psi_{\tau,m}\right) - \nabla \ln \phi_1\left( \Phi -  \Psi_{\tau,m}\right)  \right\vert_g^2 \left\vert \nabla \Phi \right\vert_g^{n-2}  \\
= \int_{\mathbb{B}_r(p)} \left\vert \nabla\left( \Phi -  \Psi_{\tau,m}\right) \right\vert_g^2 \left\vert \nabla \Phi \right\vert_g^{n-2} - \int_{\mathbb{B}_r(p)} \left\vert \nabla \Phi \right\vert_g^{n}  \left\vert \Phi -  \Psi_{\tau,m} \right\vert^2
\end{split}
\end{equation*}
so that
\begin{equation*} 
\begin{split}
\int_{\mathbb{B}_r(p)} \left\vert \nabla\left( \Phi -  \Psi_{\tau,m}\right) - \nabla \ln \phi_1\left( \Phi -  \Psi_{\tau,m}\right)  \right\vert_g^2 \left\vert \nabla \Phi \right\vert_g^{n-2}  = \int_{\mathbb{B}_r(p)} \left\vert \nabla \Psi_{\tau,m} \right\vert_g^{2} \left\vert \nabla \Phi \right\vert_g^{n-2}  \\
- \int_{\mathbb{B}_r(p)} \left\vert \nabla \Phi \right\vert_g^n \\
= \int_{\mathbb{B}_r(p)} \left\vert \nabla \Phi \right\vert_g^{n-2} \left\vert \Phi_m \right\vert^2 \left(  \left\vert \nabla \tilde{\Psi}_{\tau,m} \right\vert_g^{2} - \left\vert \nabla \tilde{\Phi}_m \right\vert_g^{2} \right)
\end{split}
\end{equation*}
Moreover, we have by similar computations that 
\begin{equation*} 
\begin{split}
\int_{\mathbb{B}_r(p)} & \left\vert \nabla\left( \tilde{\Phi}_m -  \tilde\Psi_{\tau,m}\right) - \nabla \ln \left(\tilde\Psi_{\tau,m}\right)_1\left( \tilde\Phi_m -  \tilde\Psi_{\tau,m}\right)  \right\vert_g^2  \left(\left\vert \nabla \tilde\Psi_{\tau,m} \right\vert_g^2 + \tau\right)^{\frac{n-2}{2}} \\
 = & \int_{\mathbb{B}_r(p)} \left\vert \nabla \tilde\Phi_m \right\vert_g^{2}  \left(\left\vert \nabla \tilde\Psi_{\tau,m} \right\vert_g^2 + \tau\right)^{\frac{n-2}{2}}  - \int_{\mathbb{B}_r(p)} \left\vert \nabla \tilde\Psi_{\tau,m} \right\vert_g^2 \left(\left\vert \nabla \tilde\Psi_{\tau,m} \right\vert_g^2 + \tau\right)^{\frac{n-2}{2}} \\
\end{split}
\end{equation*}
and we obtain
\begin{equation*} 
\begin{split}
A:= & \int_{\mathbb{B}_r(p)}  \left\vert \nabla\left( \tilde\Phi_m -  \tilde\Psi_{\tau,m}\right) - \nabla \ln \left(\tilde\Psi_{\tau,m}\right)_1\left( \tilde\Phi_m -  \tilde\Psi_{\tau,m}\right)   \right\vert_g^2   \left( \left\vert \nabla \tilde\Psi_{\tau,m} \right\vert_g^2 + \tau\right)^{\frac{n-2}{2}}   \\  
& + \int_{\mathbb{B}_r(p)} \left\vert \nabla\left( \Phi -  \Psi_{\tau,m}\right) - \nabla \ln \phi_1\left( \Phi -  \Psi_{\tau,m}\right)   \right\vert_g^2 \left\vert \nabla \Phi \right\vert_g^{n-2}÷   \\
\leq & \int_{\mathbb{B}_r(p)}\left(  \left\vert \nabla \tilde{\Psi}_{\tau,m} \right\vert_g^{2} - \left\vert \nabla \tilde{\Phi}_m \right\vert_g^{2}  \right)\left( \left\vert \nabla \Phi \right\vert_g^{n-2} \left\vert \Phi_m \right\vert^2- \left(\left\vert \nabla \tilde\Psi_{\tau,m} \right\vert_g^2 + \tau\right)^{\frac{n-2}{2}}  \right) \\
\end{split}
\end{equation*}
so that splitting the right-hand side into three terms, we have
\begin{equation*} 
\begin{split}
A \leq & \int_{\mathbb{B}_r(p)}\left(  \left\vert \nabla \tilde{\Psi}_{\tau,m} \right\vert_g^{2} - \left\vert \nabla \tilde{\Phi}_m \right\vert_g^{2}  \right)\left( \left\vert \nabla \tilde\Phi_m \right\vert_g^{n-2} -\left\vert \nabla \tilde\Psi_{\tau,m} \right\vert_g^{n-2}   \right) \\
& + \int_{\mathbb{B}_r(p)}\left(  \left\vert \nabla \tilde{\Psi}_{\tau,m} \right\vert_g^{2} - \left\vert \nabla \tilde{\Phi}_m \right\vert_g^{2}  \right)\left( \left\vert \nabla \Phi \right\vert_g^{n-2} \left\vert \Phi_m \right\vert^2 -  \left\vert \nabla \tilde{\Phi}_m \right\vert_g^{n-2} \right) \\
& + \int_{\mathbb{B}_r(p)}\left(  \left\vert \nabla \tilde{\Psi}_{\tau,m} \right\vert_g^{2} - \left\vert \nabla \tilde{\Phi}_m \right\vert_g^{2}  \right) \left(  \left\vert \nabla \tilde\Psi_{\tau,m} \right\vert_g^{n-2} - \left(\left\vert \nabla \tilde\Psi_{\tau,m} \right\vert_g^2 + \tau\right)^{\frac{n-2}{2}}  \right) \end{split}
\end{equation*}
and since the first right-hand term is non positive, we obtain
\begin{equation*} 
\begin{split}
A \leq & \int_{\mathbb{B}_r(p)}\left(  \left\vert \nabla \tilde{\Psi}_{\tau,m} \right\vert_g^{2} + \left\vert \nabla \tilde{\Phi}_m \right\vert_g^{2}  \right)\left\vert \left\vert \nabla \Phi \right\vert_g^{n-2} \left\vert \Phi_m \right\vert^2 -  \left\vert \nabla \tilde{\Phi}_m \right\vert_g^{n-2} \right\vert \\
& + \int_{\mathbb{B}_r(p)}\left(  \left\vert \nabla \tilde{\Psi}_{\tau,m} \right\vert_g^{2} + \left\vert \nabla \tilde{\Phi}_m \right\vert_g^{2}  \right) \left\vert  \left\vert \nabla \tilde\Psi_{\tau,m} \right\vert_g^{n-2} - \left(\left\vert \nabla \tilde\Psi_{\tau,m} \right\vert_g^2 + \tau\right)^{\frac{n-2}{2}}  \right\vert .
\end{split}
\end{equation*}
Knowing that
$$ \left\vert \nabla \Phi \right\vert^2 = \left\vert \Phi_m \right\vert^2\left\vert \nabla \tilde\Phi_m \right\vert^2 + \left\vert \nabla\left\vert \Phi_m \right\vert \right\vert^2 + \left\vert \nabla \left(\Phi-\Phi_m\right) \right\vert_g^2  $$
and that $\Phi_m$ and $\nabla \Phi_m$ converge a.e to $\Phi$ and $\nabla \Phi$, it is clear by the dominated convergence theorem that the first integral converges to $0$ as $m\to +\infty$. Thanks to the a priori estimates of Claim \ref{clapproxdeltanharm}, up to a subsequence, $\Psi$ is the limit as $(\tau,m) \to (0,+\infty)$ of $\tilde\Psi_{\tau,m}$. We easily obtain that $\Psi$ is also the limit of $\Psi_{\tau,m}$ and that the second integral in the previous inequality converges to $0$ as $(\tau,m)\to (0,+\infty)$. Moreover, if we set
$$ Z = \{ x \in \mathbb{B}_r(p) ; \left\vert \nabla \Psi \right\vert_g^2 + \left\vert \nabla \Phi \right\vert_g^2 = 0  \}, $$
we have 
$$ \left\vert \nabla\left( \Phi -  \Psi \right) - \nabla \ln \psi_1\left( \Phi -  \Psi \right)   \right\vert_g^2 + \left\vert \nabla\left( \Phi -  \Psi \right) - \nabla \ln \phi_1\left( \Phi -  \Psi \right)   \right\vert_g^2 = 0 \text{ in } \mathbb{B}_r(p)\setminus Z. $$
so that
\begin{equation*} 
\begin{split}
 \int_{\mathbb{B}_r(p)}  \left\vert \nabla\left( \Phi -  \Psi \right)  \right\vert_g^2 = \int_{\mathbb{B}_r(p)\setminus Z}   \left\vert \nabla\left( \Phi -  \Psi \right)  \right\vert_g^2 \leq 2 \int_{\mathbb{B}_r(p)\setminus Z}  \left(   \left\vert \nabla \ln \psi_1 \right\vert^2 + \left\vert \nabla \ln \phi_1 \right\vert^2 \right) \left\vert \Phi-\Psi \right\vert^2 \\
 \leq 2 C \delta_{n,g}^{\frac{2}{n}} \int_{\mathbb{B}_r(p)} \frac{\left\vert \Phi-\Psi \right\vert^2}{\left(r-\left\vert x \right\vert\right)^2} \leq  \tilde{C} \delta_{n,g}^{\frac{2}{n}}   \int_{\mathbb{B}_r(p)}  \left\vert \nabla\left( \Phi -  \Psi \right)  \right\vert_g^2 
 \end{split}
\end{equation*}
where the second inequality comes from classical weak estimates on the gradient when it satisfies a $L^{\infty}$ $\eps$-regularity property and the third one comes from a classical Hardy inequality. Letting $\delta_{n,g}$ be small enough gives that $\Phi= \Psi$.
\end{proof}

\bibliographystyle{alpha}
\bibliography{mybibfile}

\end{document}